\documentclass[12pt,usenames,dvipsnames]{amsart}

\usepackage{algorithm}
\usepackage{algpseudocode}
\usepackage{mathpazo}
\usepackage{hyperref}
\usepackage{enumitem}
\usepackage{cleveref}
\usepackage{bm}
\usepackage{comment}
\usepackage[margin=1in]{geometry}
\usepackage{enumitem}
\usepackage{calrsfs}
\DeclareMathAlphabet{\pazocal}{OMS}{zplm}{m}{n}

\usepackage{verbatim}
\usepackage[center]{caption}

\usepackage{graphicx}
\usepackage{tikz-cd}
\usepackage{multirow}

\usepackage[cmtip, all]{xy}
\usepackage{array}

\usepackage{amssymb,amsmath,latexsym,amsthm}
\newcommand{\K}{\mathbb{K}}

\DeclareMathOperator{\Hom}{Hom}

\DeclareMathOperator{\PGL}{PGL}
\DeclareMathOperator{\GL}{GL}

\DeclareMathOperator{\Spec}{Spec}

\DeclareMathOperator{\Dr}{Dr}
\DeclareMathOperator{\trop}{trop}

\newcommand{\init}{\mathsf{in}}

\DeclareMathOperator{\Gr}{Gr}

\DeclareMathOperator{\Conv}{conv}
\DeclareMathOperator{\face}{face}

\DeclareMathOperator{\Trop}{Trop}

\DeclareMathOperator{\TGr}{TGr}

\newcommand{\pC}{\pazocal{C}}

\newcommand{\field}[1]{\mathbb{#1}}

\newcommand{\Q}{\field{Q}}

\newcommand{\Z}{\field{Z}}

\newcommand{\R}{\field{R}}
\newcommand{\C}{\field{C}}

\newcommand{\A}{\field{A}}
\renewcommand{\P}{\field{P}}

\newcommand{\GG}{\mathbb{G}}

\newcommand{\pE}{\pazocal{E}}
\newcommand{\pF}{\pazocal{F}}
\newcommand{\pG}{\pazocal{G}}

\newcommand{\Sn}[1]{\mathfrak{S}_{#1}}

\newcommand{\chow}[2]{#1/\!\!\!/#2}

\newcommand{\TS}{\mathsf{TS}}

\newcommand{\pS}{\pazocal{S}}

\newcommand{\sC}{\mathsf{C}}

\newcommand{\sF}{\mathsf{F}}
\newcommand{\sG}{\mathsf{G}}
\newcommand{\sH}{\mathsf{H}}
\newcommand{\sL}{\mathsf{L}}
\newcommand{\sP}{\mathsf{P}}
\newcommand{\sQ}{\mathsf{Q}}

\newcommand{\sT}{\mathsf{T}}
\newcommand{\sU}{\mathsf{U}}
\newcommand{\sV}{\mathsf{V}}
\newcommand{\sW}{\mathsf{W}}

\newcommand{\op}{\mathsf{op}}

\newcommand{\vertices}{\mathsf{vert}}
\newcommand{\edge}{\mathsf{edge}}

\newcommand{\Csch}{\C\operatorname{--sch}}

\newcommand{\se}{\mathsf{e}}
\newcommand{\Zone}{\langle \mathbf{1} \rangle}

\newcommand{\sa}{\mathsf{a}}
\newcommand{\su}{\mathsf{u}}
\newcommand{\sv}{\mathsf{v}}
\newcommand{\sw}{\mathsf{w}}

\newcommand{\col}{\mathsf{col}}

\newcommand{\Leaf}{\mathsf{Leaf}}

\newcommand{\spMat}{\sQ_\mathsf{sp}}
\newcommand{\spCone}{\pC_\mathsf{sp}}
\newcommand{\spw}{\sw_\mathsf{sp}}

\newcommand{\matroidal}{\mathsf{mat}}

\newcommand{\pQ}{\pazocal{Q}}

\DeclareMathOperator{\semigp}{semigp}

\DeclareMathOperator{\coef}{coef}

\newtheorem{theorem}{Theorem}[section]
\newtheorem{lemma}[theorem]{Lemma}
\newtheorem{proposition}[theorem]{Proposition}
\newtheorem{corollary}[theorem]{Corollary}

\newtheorem*{theorem*}{Theorem}

\theoremstyle{definition}

\newtheorem{example}[theorem]{Example}
\newtheorem{algm}[theorem]{Algorithm}

\theoremstyle{remark}
\newtheorem{remark}[theorem]{Remark}

\numberwithin{equation}{section}
\numberwithin{table}{section}
\numberwithin{figure}{section}

\newcommand\shorttitle[1]{\renewcommand\@shorttitle{#1}}

\title{The Grassmannian of $3$-planes in $\C^{8}$ is sch\"on}

\author{Daniel Corey}
\author{Dante Luber}

\begin{document}

\maketitle

\begin{abstract}
    We prove that the open subvariety $\Gr_0(3,8)$ of the Grassmannian $\Gr(3,8)$ determined by the nonvanishing of all Pl\"ucker coordinates is sch\"on, i.e., all of its initial degenerations are smooth. Furthermore, we find an initial degeneration that has two connected components, and show that the remaining initial degenerations, up to symmetry, are irreducible. As an application, we prove that the Chow quotient of $\Gr(3,8)$ by the diagonal torus of $\PGL(8)$ is the log canonical compactification of the moduli space of $8$ lines in $\P^2$, resolving a conjecture of Hacking, Keel, and Tevelev. Along the way we develop various techniques to study finite inverse limits of schemes.
    
 (secondary)
 
\end{abstract}

\section{Introduction}

A closed subvariety of an algebraic torus is \textit{sch\"on} if all of its initial degenerations---flat degenerations arising via Gr\"obner theory---are smooth. This notion was introduced by Tevelev in his influential paper \cite{Tevelev}, and admits this characterization by Helm and Katz \cite{HelmKatz}.  Sch\"on subvarieties of tori satisfy many desirable properties. Their tropical compactifications are sch\"on compactifications, in particular they are normal and have toroidal singularities. Notions from birational geometry, like log minimality and ampleness of the log canonical divisor, admit tropical characterizations for sch\"on subvarieties and their sch\"on compactifications, respectively \cite{HackingKeelTevelev2009}. Some notable examples include nondegenerate hypersurfaces \cite{VarchenkoNewton, VarchenkoZeta} (whose study predates the notion of sch\"onness), complements of hyperplane arrangements, open del Pezzo surfaces, and moduli spaces of marked del Pezzo surfaces \cite{HackingKeelTevelev2009}. 

Consider the Grassmannian $\Gr(r,n)$ of $r$-planes in $\K^{n}$, where $\K$ is an algebraically closed field, and let $\Gr_0(r,n)$ be the open locus of $\Gr(r,n)$ defined by the nonvanishing of all Pl\"ucker coordinates. Its tropicalization $\TGr_0(r,n)$ parameterizes all tropical linear spaces realizable over a valued field extension of $\K$ \cite[Theorem~4.3.17]{MaclaganSturmfels2015}, so $\TGr_0(r,n)$, and therefore the initial degenerations of $\Gr_0(r,n)$, depend on the characteristic of $\K$. For this reason, we assume that $\K=\C$ throughout. The Grassmannian $\Gr_0(r,n)$ is known to be sch\"on when $(r,n) = (2,n)$ \cite{Tevelev}, $(3,6)$ \cite{Luxton} and $(3,7)$ \cite{CoreyGrassmannians}.  

\begin{theorem}
\label{thm:schoenIntro}
The Grassmannian $\Gr_0(3,8)$ is sch\"on.  
\end{theorem}

\begin{remark}
The initial degenerations of Grassmannian $\Gr_0(3,8)$ are isomorphic, up to a torus factor, to those of the moduli space $X(3,8)$ parameterizing isomorphism classes of $8$ hyperplanes in $\P^{2}$ lying in general position. See Formula \ref{eq:initStratum}. In particular, $\Gr_0(3,8)$ is sch\"on if and only if $X(3,8)$ is sch\"on. In \cite{Schock}, Schock develops a theory of quasilinear varieties. Schock proves that a quasilinear variety $X$ is sch\"on and the strata of any sch\"on compactification of $X$ are irreducible. Schock also proposes a proof that $X(3,8)$ is quasilinear, and hence sch\"on, but this argument contains a critical gap. By Theorem \ref{thm:2ConnectedComponentsInDegIntro}, $X(3,8)$ has an initial degeneration that is disconnected, and hence $X(3,8)$ is not quasilinear.
\end{remark}

The tropical Grassmannian  $\TGr_0(r,n)$ lies in a larger space $\Dr(r,n)$ called the \textit{Dressian}, which parameterizes \textit{all} tropical linear spaces \cite{HerrmannJensenJoswigSturmfels, Speyer2008}. Initially studied by Lafforgue \cite{Lafforgue2003}, to a $\sw \in \Dr(r,n)$ is associated a finite inverse limit $\Gr(\sw)$ of thin Schubert cells parameterized by the induced matroidal subdivision of the hypersimplex. When $\sw\in \TGr_0(r,n)$, we may also form the $\sw$-initial degeneration $\init_{\sw}\Gr_0(r,n)$. We use in an essential way the main theorem of \cite{CoreyGrassmannians}, which asserts that, for any $\sw\in \TGr_0(r,n)$, there is a closed immersion
\begin{equation*}
    \init_{\sw}\Gr_0(r,n) \hookrightarrow \Gr(\sw). 
\end{equation*}
We deduce Theorem \ref{thm:schoenIntro} from a stronger result, that when $(r,n) = (3,8)$, the above mapa are isomorphisms, and the limits $\Gr(\sw)$ are smooth. Thus, while the statement of this theorem is geometric, the proof relies on techniques from matroid theory, commutative algebra (computing the coordinate rings of thin Schubert cells and morphisms between them) and polyhedral complexes (regular subdivisions of the hypersimplex and their tight-spans).  Given the connection between initial degenerations of $\Gr_0(r,n)$ and thin Schubert cells (the collection of which satisfies Murphy's Law by Mn\"ev universality), it is expected that $\Gr_0(r,n)$ is not, in general, sch\"on, although this remains an open problem.

\begin{figure}[h]
    \includegraphics[height=4cm]{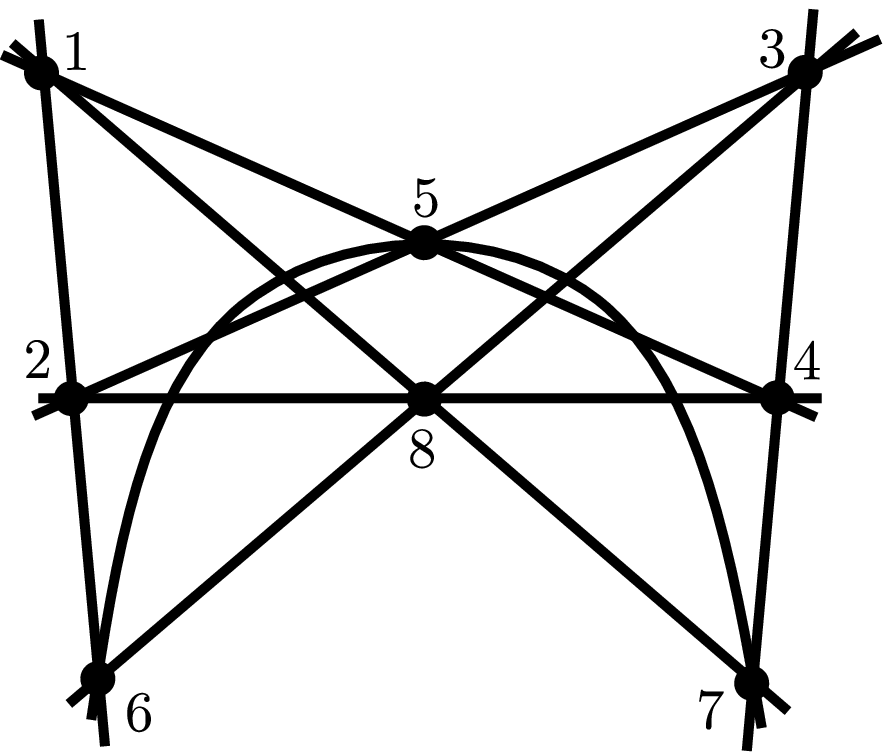}
    \caption{The matroid $\spMat$}
    \label{fig:spMatroidIntro}
\end{figure}

Next, we study the connectedness properties of the initial degenerations of $\Gr_{0}(3,8)$.  Interestingly, they are not all connected. Let $\spMat$ be the rank-3 matroid from Figure \ref{fig:spMatroidIntro}, i.e., each 3-element set not connected by a (possibly curved) line is a basis. This is the unique (up to the action of the symmetric group $\Sn{8}$) sparse-paving rank-3 matroid on $[8]$ with 8 nonbases (the maximum possible).  We show in \S \ref{sec:38mat} that $\spMat$ is not realizable over $\R$, but over $\Q(\sqrt{-3})$. Denote by $\spCone$ the cone of $\TGr_0(3,8)$ containing the corank vector of $\spMat$ in its relative interior.  

\begin{theorem}
\label{thm:2ConnectedComponentsInDegIntro}
For any $\sw$ in the relative interior of a cone in the $\Sn{8}$--orbit of $\spCone$, the initial degeneration $\init_{\sw} \Gr_{0}(3,8)$ has 2 connected components.  
\end{theorem}

\noindent Up to $\Sn{8}$--symmetry, this is the only non-connected initial degeneration of $\Gr_0(3,8)$.

\begin{theorem}
\label{thm:connectedInitDegIntro}
    If $\sw\in \TGr_0(3,8)$ is not in the relative interior of a cone in the $\Sn{8}$--orbit of $\pC_{\mathsf{sp}}$, then $\init_{\sw}\Gr_0(3,8)$ is irreducible. 
\end{theorem}

We apply these results to study the Chow quotient compactification of $X(r,n)$, the moduli space of projective equivalence classes of $n$ marked hyperplanes of $\P^{r-1}$ in linear general position. The diagonal torus of $\GL(n)$ acts on $\C^n$ by scaling coordinates, and this induces an action of the diagonal torus $H\subset \PGL(n)$ on $\Gr(r,n)$. This action is free on $\Gr_0(r,n)$, and $X(r,n)$ is identified with the quotient $\Gr_0(r,n)/H$ by the Gelfand-MacPherson correspondence \cite{GelfandMacPherson}. The space $X(2,n)$ is also known as $M_{0,n}$, the moduli space of smooth rational $n$-marked curves.

A natural candidate to compactify $X(r,n)$ is given by the Chow quotient $\chow{\Gr(r,n)}{H}$, studied by Kapranov \cite{Kapranov1993}, where he demonstrates that the Grothendieck-Knudsen moduli space of genus 0, stable, $n$-marked curves $\overline{M}_{0,n}$ is isomorphic to $\chow{\Gr(2,n)}{H}$. This is also the log-canonical compactification of $M_{0,n}$ \cite{KeelMcKernan}. Keel and Tevelev show that $\chow{\Gr(r,n)}{H}$ is usually not log canonical, failing already for $(r,n) = (3,9)$ \cite{KeelTevelev2006}. With Hacking, they conjecture [\textit{loc. cit.}, Conjecture 1.6] that it is log canonical when $(r,n) = (2,n)$, $(3,6)$, $(3,7)$, and $(3,8)$. The conjecture is known in the first 4 cases, by Keel and McKernan (mentioned above), Luxton \cite{Luxton}, and the first author \cite{CoreyGrassmannians}, respectively. Using Theorems \ref{thm:schoenIntro}, \ref{thm:2ConnectedComponentsInDegIntro} and \ref{thm:connectedInitDegIntro}, and a result of Schock \cite[Proposition~7.9]{Schock}, we prove the $(3,8)$ case.

\begin{theorem}
\label{thm:chowQuotientIntro}
The Chow quotient $\chow{\Gr(3,8)}{H}$ is the log canonical compactification of $X(3,8)$.
\end{theorem}

Here is an outline of the paper. In \S \ref{sec:inverseLimitsTSC}, we review the relationship between initial degenerations of $\Gr_0(r,n)$ and finite inverse limits of thin Schubert cells induced by a matroidal subdivision of the hypersimplex. We consider the general setup of diagrams of schemes in \S\ref{sec:toolsForLimits}, and develop various strategies to determine when their inverse limits are smooth and irreducible. While not all morphisms between thin Schubert cells of $\Gr(3,8)$ are smooth and dominant, we develop general criteria in \S \ref{sec:computationsTSC} to detect when they are smooth and dominant with connected fibers. The disconnected initial degeneration of Theorem \ref{thm:2ConnectedComponentsInDegIntro} is studied in \S \ref{sec:reducibleInitDeg}. Theorems \ref{thm:schoenIntro} and \ref{thm:connectedInitDegIntro} are proved in \S\ref{sec:proofThms1And3}. We use the techniques from \S\S \ref{sec:toolsForLimits}-\ref{sec:computationsTSC} to prove that each inverse limit $\Gr(\sw)$ smooth of dimension 15 and, excluding the case from Theorem \ref{thm:2ConnectedComponentsInDegIntro}, irreducible. For any individual $\sw$, once one has the regular subdivision, one may carry out this verification by hand, as we illustrate in several examples throughout that section. As there are $57\,344$ cases to consider, we must use software. Finally, we discuss the Chow quotient in \S\ref{sec:Chow} and prove Theorem \ref{thm:chowQuotientIntro}. 

\subsection*{Code} Throughout \S\ref{sec:proofThms1And3}, we use \texttt{polymake.jl}  \cite{polymake,polymakeJL} and OSCAR \cite{OSCAR-book,OSCAR}, both of which run using \texttt{julia} \cite{BezansonEdelmanKarpinskiShah}. The code can be found at the following github repository:

\begin{center}
    \url{https://github.com/dcorey2814/Gr38Schoen}
\end{center}

\subsection*{Acknowledgements} We thank Benjamin Schr\"oter for sharing the data of the tropicalization of $\Gr_0(3,8)$ with us. We also thank Michael Joswig, Lars Kastner, Benjamin Lorenz, Sam Payne, and Antony Della Vecchia for helpful conversations. DC is supported by the SFB-TRR project ``Symbolic Tools in Mathematics and their Application'' (project-ID 286237555), and DL is supported by "Facets of Complexity” (GRK 2434, project-ID 385256563).

\section{Inverse limits of thin Schubert cells} 
\label{sec:inverseLimitsTSC}

\subsection{Matroidal subdivisions}
\label{sec:matroidalSubdivisions}
We assume that the reader is familiar with matroids; a good general reference is \cite{Oxley}. While there are many ways of characterizing matriods, the definition via \textit{bases} is most relevant. A matroid of rank $r$ on $[n]$, called an $(r,n)$-\textit{matroid}, is a nonempty subset $\sQ$ of $\binom{[n]}{r}$ that satisfies the basis-exchange axiom.

Let $N = \Z^{n}/\Zone$ and $M = \Hom(N,\Z)$. We use the standard abbreviations $M_{A} = M \otimes_{\Z} A$ and  $N_{A} = N \otimes_{\Z} A$ where $A$ is a $\Z$-module.  Denote by $\epsilon_{1},\ldots,\epsilon_n \in N$  the images of the standard basis vectors under the projection $\Z^{n} \to N$, and $\epsilon_{\lambda} = \epsilon_{i_1} + \cdots + \epsilon_{i_r} $ whenever $\lambda = \{i_1,\ldots,i_r\} \in \binom{[n]}{r}$.  Denote by $\epsilon_i^{*}$ and $\epsilon_{\lambda}^{*}$ in $M$ the duals of $\epsilon_i$ and $\epsilon_{\lambda}$, respectively.

Given $\sQ \subset  \binom{[n]}{r}$,  set
\begin{equation*}
    \Delta(\sQ) = \Conv \{\epsilon_{\lambda}^{*} \, : \, \lambda \in \sQ\} \subset M_{\R}
\end{equation*}
When $\sQ$ is a matroid, $\Delta(\sQ)$ is the \textit{matroid polytope} of $\sQ$. An important special case is the \textit{uniform matroid} $\sQ = \binom{[n]}{r}$, and its polytope $\Delta(\sQ)$ is the $(r,n)$-\textit{hypersimplex} which we denote $\Delta(r,n)$. 

Let  $N(\sQ) = \Z^{\sQ}/\Zone$ and $M(\sQ) = \Hom(N(\sQ),\Z)$. Denote by $\{\se_{\lambda}\, : \, \lambda \in \sQ\}$ the images of the standard basis vectors under $\Z^{\sQ} \to N(\sQ)$. For $\sw \in N(\sQ)_{\R} := N(\sQ) \otimes_{\Z} \R $, denote by $\pQ(\sw)$ the \textit{regular subdivision} of $\Delta(\sQ)$ induced by $\sw$; this is the polyhedral complex obtained by lifting, for each $\lambda \in \sQ$, the vertex $\epsilon_{\lambda}^{*}$ of $\Delta(\sQ)$ to height $\sw_{\lambda}$ in $M_{\R} \times \R$, then projecting the lower faces of this lifted polytope back down to $M_{\R}$. See \cite[Chapter~2]{DeLoeraRambauSantos} for a precise treatment. 

If $\sQ$ is a matroid, then the subdivision $\pQ(\sw)$ is \textit{matroidal} if each cell in $\pQ(\sw)$ is a matroid polytope.  The \textit{Dressian} of $\sQ$ \cite{HerrmannJensenJoswigSturmfels} is the set
\begin{equation*}
    \Dr(\sQ) = \{\sw \in N(\sQ)_{\R} \, : \, \pQ(\sw) \; \text{ is matroidal} \}.
\end{equation*}
 This is the support of a polyhedral fan in $N(\sQ)$.

\subsection{Tropicalization}

Consider a pair of finite-rank lattices $M,N$ with a perfect pairing 
\begin{equation*}
    (u, v) \to \langle u, v \rangle
\end{equation*}
Let $T = \Spec(\C[M])$ be the torus whose character and cocharacter lattices are $M$ and $N$, respectively. Denote by $x^{\su} \in \C[M]$ the monomial corresponding to $\su\in M$.  Suppose $X\subset T$ is a closed subvariety defined by a prime ideal $I\subset \C[M]$. Given $\sw \in N_{\R}$, the \textit{$\sw$-initial form} of $f\in \C[M]$ is
\begin{equation*}
    \init_{\sw} f = \sum_{\substack{\langle \su, \sw \rangle \\ \text{is minimal}} } c_{\su} x^{\su} \hspace{20pt} \text{where} \hspace{20pt} f = \sum_{\su} c_{\su} x^{\su} 
\end{equation*}
The \textit{$\sw$-initial ideal} of $I$ is $\init_{\sw} I = \langle \init_{\sw} f \, : \, f\in I\rangle$. The \textit{tropicalization} of $X$ is
\begin{equation*}
    \Trop(X) = \{\sw \in N_{\R} \, : \, \init_{\sw} I  \neq \langle 1 \rangle \}.
\end{equation*}
This set is the support of a rational polyhedral fan in $N_{\R}$. Given $\sw \in \Trop(X)$, the \textit{$\sw$-initial degeneration} of $X$, denoted $\init_w X$, is the subscheme of $T$ defined by the ideal $\init_{\sw} I$.

\subsection{The Grassmannian}
\label{sec:Grassmannian}
As a set, the Grassmannian $\Gr(r,n)$ consists of all $r$-dimensional linear subspaces of $\C^{n}$. It may be realized as a projective variety by the Pl\"ucker embedding:
\begin{align*}
    \Gr(r,n) \hookrightarrow \P(\wedge^{r}\C^n ) \cong \P^{\binom{n}{r} - 1}  \hspace{20pt} F \mapsto \wedge^r F = [\zeta_{\lambda}(F) \, : \, \lambda \in \textstyle{\binom{[n]}{r}}]
\end{align*}
The homogeneous coordinates $\zeta_{\lambda}(F)$ are called the \textit{Pl\"ucker coordinates} of $F$. 
Concretely, if $F \subset \C^n$ is the row span of the full-rank $(r\times n)$-matrix $A$, then $\zeta_{\lambda}(F)$ is the determinant of the $(r\times r)$-submatrix of $A$ whose columns are indexed by $\lambda$.

Given a matroid $\sQ \subset \binom{[n]}{r}$, its \textit{thin Schubert cell} is the scheme-theoretic intersection 
\begin{equation*}
    \Gr(\sQ) = \Gr(r,n) \cap T(\sQ)
\end{equation*}
where $T(\sQ) = \GG_m^{\sQ} / \GG_m$ is the locally-closed coordinate stratum of $\P^{\binom{n}{r}-1}$ whose nonzero coordinates are exactly those belonging to $\sQ$. Of particular interest is the thin Schubert cell of the uniform matroid $\sQ = \binom{[n]}{r}$, which we denote by  $\Gr_0(r,n)$; this is the open locus of $\Gr(r,n)$ determined by the nonvanishing of all Pl\"ucker coordinates. Throughout, we use the abbreviations 
\begin{equation*}
\TGr(\sQ) = \Trop(\Gr(\sQ)) \hspace{20pt} \text{and} \hspace{20pt} \TGr_0(r,n) = \Trop(\Gr_0(r,n))    
\end{equation*}
The set $\TGr(\sQ) \subset N(\sQ)_{\R}$ is preserved under translation by the $(n-1)$-dimensional linear subspace $L_{\R}$, where $L\subset N(\sQ)$ is the subgroup  
\begin{equation}
\label{eq:lineality}
    L = \left\langle  \sum_{i\in \lambda  \in \sQ} \se_{\lambda}, \;\;  \sum_{i\notin \lambda \in \sQ} \se_{\lambda} \, : \, i\in [n] \right\rangle. 
\end{equation}
We have an inclusion $\TGr(\sQ) \subset \Dr(\sQ)$ by \cite[Proposition~2.2]{Speyer2008} in the uniform matroid case and \cite[Lemma~4.4.6]{MaclaganSturmfels2015} in general. Of particular importance is that \textit{$\pQ(\sw)$ is matroidal for any $\sw \in \TGr(\sQ)$.}

Denote by $\pG(r,n)$ the Gr\"obner fan structure of $\TGr_0(r,n)$ and $\pS_{\matroidal}(r,n)$ the secondary fan structure on $\Dr(r,n)$. By \cite[Theorem~5.4]{Tevelev},  $\pG(r,n)$ is a subfan of a refinement of $\pS_{\matroidal}(r,n)$.  For $(r,n) = (2,n)$,  $(3,6)$ \cite{SpeyerSturmfels2004a}, $(3,7)$ \cite{HerrmannJensenJoswigSturmfels}, or $(3,8)$ \cite{BendleBoehmRenSchroter}, there is a subfan $\pS_{\trop}(r,n)$ of $\pS_{\matroidal}(r,n)$ whose support is $\TGr_0(r,n)$.

The \textit{face order} on the set of $(r,n)$--matroids is defined by $\sP \leq \sQ$  whenever  $\Delta(\sP)$  is a face of $\Delta(\sQ)$.
For each pair $\sP \leq \sQ$, the coordinate projection $T(\sQ) \to T(\sP)$ induces a morphism $\varphi_{\sQ,\sP}: \Gr(\sQ) \to \Gr(\sP)$ \cite[Proposition~I.6]{Lafforgue2003}. Given $\sw \in \TGr(\sQ)$, the regular subdivision $\pQ(\sw)$ of $\Delta(\sQ)$ is matroidal in the sense of \S \ref{sec:matroidalSubdivisions}. The assignment
\begin{equation*}
  \sP \mapsto \Gr(\sP) \hspace{20pt} \sP' \leq  \sP \mapsto \varphi_{\sP,\sP'}: \Gr(\sP) \to \Gr(\sP')  
\end{equation*}
defines a diagram of type $\pQ(\sw)$ in the category of affine $\C$-schemes, and therefore we may form its inverse limit:
\begin{equation*}
    \Gr(\sw) := \varprojlim_{\pQ(\sw)} \Gr. 
\end{equation*}
The following theorem \cite[Theorem~1.1]{CoreyGrassmannians} is fundamental to the proofs of the main theorems in this paper. 
\begin{theorem}
\label{thm:closedImmersion}
For any $\sw\in \TGr_0(r,n)$, there is a closed immersion 
\begin{equation*}
    \init_{\sw} \Gr_0(r,n) \hookrightarrow \Gr(\sw).
\end{equation*}
\end{theorem}
\noindent We thus obtain the following.
\begin{corollary}
\label{cor:limit2Init}
Suppose $\sw\in \TGr_0(r,n)$. If $\Gr(\sw)$ is smooth and irreducible of dimension $r(n-r)$, then $\init_{\sw}\Gr_0(r,n)$ is also smooth and irreducible. 
\end{corollary}

\begin{proof}
Being a limit of a flat degeneration of $\Gr_0(r,n)$, the initial degeneration $\init_{\sw} \Gr_0(r,n)$ is an affine scheme of dimension $r(n-r)$. The inverse limit $\Gr(\sw)$ also an affine scheme, so if it is smooth, irreducible, and has the same dimension as $\init_{\sw} \Gr(r,n)$, then the closed immersion from Theorem \ref{thm:closedImmersion} is an isomorphism \cite[Proposition~A.8]{CoreyGrassmannians}. 
\end{proof}

 In the course of proving Theorems \ref{thm:schoenIntro}, \ref{thm:2ConnectedComponentsInDegIntro}, and \ref{thm:connectedInitDegIntro}, we show that the closed immersion in Theorem \ref{thm:closedImmersion} is an isomorphism for each $\sw\in \TGr_0(3,8)$. This provides a conceptual explanation of \cite[Theorem~4.8]{BendleBoehmRenSchroter}, which states that two vectors $\sw_1,\sw_2 \in \TGr_0(3,8)$ lie in the same cone of $\pS_{\trop}(3,8)$ if and only if $\init_{\sw_1}\Gr_0(3,8) = \init_{\sw_2}\Gr_0(3,8)$. In contrast, the cones in the Gr\"obner fan parameterize the various initial ideals of the homogeneous Pl\"ucker ideal.

\subsection{Limits over graphs and tight spans}
\label{sec:limitsGraphsTS}

Computing the inverse limit $\Gr(\sw)$ seems like an unmanegable task given the size of the full face poset of $\pQ(\sw)$. It turns out that the codimension 0 and 1 cells that meet the relative interior of $\Delta(\sQ)$ are sufficient to compute this limit. For added flexibility, we describe a collection of subposets of $\pQ(\sw)$ that carry enough information to determine the inverse limit $\Gr(\sw)$. 

Let $\pQ(\sw)$ be a regular subdivision of a lattice polytope $\Delta(\sQ) \subset M_{\R}$ induced by the lifting function $\sw \in N(\sQ)$. The \textit{dual graph} of $\pQ(\sw)$, denoted by $\Gamma(\sw)$, is the graph that has a vertex $v_{\sP}$ for each codimension-0 face $\sP \in \pQ(\sw)$, and the vertices $v_{\sP}$, $v_{\sP'}$ are connected by an edge if and only if $\sP$ and $\sP'$ meet along a common codimension-$1$ face.

The \textit{tight span} of $\pQ(\sw)$, denoted $\TS(\sw)$, is the polyhedral complex that has a cell of dimension $\dim \Delta(\sQ) - k$ for each $k$-dimensional cell of $\pQ(\sw)$ meeting the relative interior of $\Delta(\sQ)$, and with face identifications opposite those of $\pQ(\sw)$. See \cite{Herrmann} for a precise treatment. The dual graph is exactly the 1-skeleton of the tight-span. Given any polyhedral complex $\Sigma$ with $\Gamma(\sw) \subset \Sigma \subset \TS(\sw)$, we may view $\Sigma^{\op}$ as a subposets of $\pQ(\sw)$, and hence we may form the limits of thin Schubert cells over this diagram.  We denote such a limit by $\varprojlim_{\Sigma} \Gr$.

\begin{proposition}
\label{prop:limitTSGamma}
Given any polyhedral complex $\Sigma$ with $\Gamma(\sw) \subset \Sigma \subset \TS(\sw)$, the natural morphisms
\begin{equation*}
      \Gr(\sw) \xrightarrow{\Phi} \varprojlim_{\Sigma} \Gr \xrightarrow{\Psi}  \varprojlim_{\Gamma(\sw)} \Gr  
\end{equation*} 
are isomorphisms. 
\end{proposition}

\begin{proof}
The composition $\Psi \circ \Phi$ is an isomorphism by a result of Cueto \cite[Proposition~C.12]{CoreyGrassmannians}; denote by $\Theta:\varprojlim_{\Gamma(\sw)} \Gr\to\Gr(\sw)$ its inverse. One readily verifies that $\Phi \circ \Theta$ is the inverse to $\Psi$.
\end{proof}

\section{Tools to study finite inverse limits of schemes}
\label{sec:toolsForLimits}

We gather in this section a number of techniques that will allow us to study inverse limits of diagrams of schemes coming from contractible polyhedral complexes, e.g., tight-spans of matroidal subdivisions of the hypersimplex. More precisely, we want techniques to prove that a finite inverse limit is smooth and irreducible just by understanding the spaces and maps in the inverse limit system. With these techniques in hand, we avoid the necessity to directly compute the inverse limit.  

The key facts underlying most arguments in this section are contained in the following proposition. Recall that a \textit{SDC-morphism} of $\C$-schemes is a morphism of schemes  $\varphi:X\to Y$ that is smooth, dominant, and its nonempty fibers are connected.

\begin{proposition}
\label{prop:smoothIrreducibleFiberProduct}
Suppose we have morphisms of equidimensional $\C$-schemes $X\to W$ and $Y\to W$.
\begin{enumerate}
    \item If $Y$ is smooth and irreducible, and $X\to W$ is an SDC-morphism, then $X\times_W Y$, if nonempty, is smooth and irreducible. 
    \item If $Y$ is smooth with $k$-connected components, and $X\to W$ is smooth and surjective with connected fibers, then  $X\times_W Y$ is smooth with $k$ connected components. 
\end{enumerate}
In either case, the dimension of $X\times_W Y$ is
\begin{equation*}
    \dim X\times_W Y = \dim X + \dim Y - \dim W.
\end{equation*}
\end{proposition}

\begin{proof}
Part (1) is \cite[Proposition~A.2]{CoreyGrassmannians}, so consider (2). Smoothness, surjectivity, and connectivity of fibers are all preserved by base change, so  $X\times_W Y \to Y$ has these properties. The fiber product $X\times_W Y$ is a smooth $\C$-scheme as $X\times_W Y \to \Spec \C$ is the composition of two smooth morphisms. By  \cite[\href{https://stacks.math.columbia.edu/tag/0378}{Tag 0378}]{stacks-project}, the morphism $X\times_W Y \to Y$ induces a bijection between the connected components of these spaces, and hence  $X\times_W Y$ has $k$ connected components. The dimension formula is standard, see, e.g., \cite[\S III.9]{Hartshorne}
\end{proof}

Throughout this section, let $\Sigma$ be a finite connected polyhedral complex, and $\Gamma$ its 1-skeleton. We regard $\Sigma$ as a quiver induced by its face order. Thus, we have an arrow $\sC' \to \sC$ whenever $\sC'$ is a \textit{facet} of $\sC$ in $\Sigma$. A \textit{diagram} of type $\Sigma$ in $\Csch$, the category of $\C$-schemes, is simply a functor $Z:\Sigma \to \Csch$. As finite inverse limits exist in the category $\Csch$ (indeed, it is enough to have fiber products and a terminal object, see \cite[Proposition~5.21]{Awodey}), we may form the inverse limit $\varprojlim_{\Sigma} Z$. 

\begin{proposition}
\label{prop:limitTreeAllButOne}
Suppose $Z: \Gamma \to \Csch$ is a diagram over a tree $\Gamma$, and that  $Z(v)$ is smooth and irreducible for all $v\in V(\Gamma)$. Furthermore, suppose that $Z(v \to e)$ is an SDC-morphism for all $v\in V(\Gamma)$ and $e\in E(\Gamma)$ adjacent to $v$, except for possibly one vertex $v_0$. Then $\varprojlim_{\Gamma} Z$ is smooth and irreducible of dimension
\begin{equation}
\label{eq:dimTree}
    \dim \varprojlim_{\Gamma} Z =  -\!\sum_{e\in E(\Gamma)} \dim Z(e) \;\; + \sum_{v\in V(\Gamma)} \dim Z(v) 
\end{equation}

\end{proposition}

\begin{proof}
We proceed by induction on the number of edges. The proposition is trivial if $\Gamma$ has no edges, so assume that $\Gamma$ has $n$ edges and that the proposition is true for graphs with fewer than $n$ edges. Suppose $v\neq v_0$ is connected to $v_0$ by the edge $e$. Then $Z(v)\times_{Z(e)} Z(v_0)$ is smooth and irreducible by Proposition \ref{prop:smoothIrreducibleFiberProduct}. Let $Z'$ be the diagram on $\Gamma/e$ where $Z'(v) = Z(v)\times_{Z(e)} Z(v_0)$ at the contracted vertex, and $Z' = Z$ otherwise. By \cite[Proposition~A.5]{CoreyGrassmannians} we have that $\varprojlim_{\Gamma} Z \cong \varprojlim_{\Gamma/e}Z'$. As $Z'$ is a diagram on a graph with fewer edges and satisfies the hypotheses of the proposition, the limit $\varprojlim_{\Gamma/e}Z'$ is smooth and irreducible by the inductive hypothesis. 
\end{proof}

\subsection{Removing leaves} \label{sec:removeLeaf} 
A \textit{leaf} of $\Sigma$ is a vertex $v$ of $\Sigma$ that is a leaf vertex of $\Gamma$. A \textit{leaf-pair} is a pair $(v,e)$ such that $v$ is a leaf and $e$ is its adjacent edge. Let $\Sigma_{\sL} \subset \Sigma$ be the subcomplex obtained by removing all leaf-pairs from $\Sigma$. 

\begin{proposition}
\label{prop:removeLeavesLim}
If $\varprojlim_{\Sigma_{\sL}} Z$ is smooth and irreducible and $Z(v\to e)$ is an SDC-morphism for each leaf pair $(v,e)$ then $\varprojlim_{\Sigma} Z$ is smooth and irreducible of dimension
\begin{equation}
    \label{eq:dimDeLeaf}
    \dim \varprojlim_{\Gamma} Z = \dim \varprojlim_{\Sigma_{\sL}} Z \;\;- \sum_{\substack{e \text{ leaf} \\
    \text{edge}}} \dim Z(e)\;\;+ \sum_{\substack{v \text{ leaf} \\
    \text{vertex}}} \dim Z(v)
\end{equation}
 \end{proposition}
 \begin{proof}
 Let $\Gamma'$ be the star-shaped graph obtained by contracting $\Sigma_{\sL}$ to a point. Let $Z'$ be the diagram on $\Gamma'$ defined by $Z$ on all leaf-pairs, and $\varprojlim_{\Sigma_{\sL}} Z$ on the newly-formed vertex. Set
\begin{equation*}
    Y_{\Leaf} = \prod_{\substack{v \text{ leaf} \\ \text{vertex}}} Z(v) \hspace{20pt}  Y_{\Leaf}' = \prod_{\substack{e \text{ leaf} \\ \text{edge}}} Z(e)
\end{equation*}
and note that $Y_{\Leaf} \to Y_{\Leaf}'$ is an SDC-morphism since each $Z(v\to e)$ is an SDC-morphism.  By \cite[Proposition~A.5]{CoreyGrassmannians} and Proposition \ref{prop:limitTSGamma}, we have
\begin{equation*}
    \varprojlim_{\Sigma} Z \cong \varprojlim_{\Gamma'} Z'
\end{equation*}
The proposition now follows from Proposition \ref{prop:limitTreeAllButOne}. 
 \end{proof}
 
Observe that $\Sigma_{\sL} \subset \Sigma$ may still have leaves. The process of iteratively removing leaf-pairs must terminate, and the result is the subcomplex $\Sigma_{\mathsf{Br}} \subset \Sigma$, which is nonempty provided  $\Sigma$ is not a tree. A maximal connected subgraph of $\Sigma \setminus \Sigma_{\mathsf{Br}}$ is called a \textit{branch} of $\Sigma$. Any vertex in a branch is called a \textit{branch vertex} and any edge contained in a branch, or connecting a branch to the rest of $\Sigma$, is called a \textit{branch edge}. The following proposition follows from Proposition \ref{prop:removeLeavesLim} and induction.

\begin{proposition}
\label{prop:removeBranchLim}
If $\varprojlim_{\Sigma_{\mathsf{Br}}} Z$ is smooth and irreducible and $Z(v\to e)$ is an SDC-morphism for each pair $(v,e)$ of a branch vertex $v$ and adjacent edge $e$, then $\varprojlim_{\Sigma} Z$ is smooth and irreducible of dimension
\begin{equation}
    \label{eq:dimDeBranch}
        \dim \varprojlim_{\Gamma} Z = \dim \varprojlim_{\Sigma_{\mathsf{Br}}} Z  - \sum_{\substack{e \text{ branch} \\
    \text{edge}}} \dim Z(e) + \sum_{\substack{v \text{ branch} \\
    \text{vertex}}} \dim Z(v).
\end{equation}
\end{proposition}

\subsection{Removing fins}\label{sec:removeFins}
 Suppose $\sF$ is a closed 2-dimensional cell of $\Sigma$ with $k$ vertices ($k\geq 3$). We say that $\sF$ is a \textit{fin} if its intersection with $\overline{\Sigma \setminus \sF}$ is a path of edge-length $\ell$ with  $1\leq \ell \leq k-2$; this path is called the \textit{connecting path} of $\sF$. See Figure \ref{fig:attachFin} for an illustration. Denote by $V(\sF)$ the vertices of $\sF$ and $E(\sF)$ its edges.
 A vertex $v\in V(\sF)$, resp. an edge $e\in E(\sF)$, is \textit{exposed} if $v\notin \overline{\Sigma\setminus \sF}$, resp. $e\not\subset \overline{\Sigma\setminus \sF}$. Denote by
\begin{equation*}
\pE_{\vertices}(\sF) = \{v \in V(\sF) \, : \, v \text{ is exposed}\} \hspace{20pt}   \pE_{\edge}(\sF) = \{e\in E(\sF)  \, : \, e \text{ is exposed}\}
\end{equation*}
Order the vertices of $\sF$ cyclically $v_1,\ldots,v_{k}$ so that the first $k-\ell-1$ are the exposed vertices. Let $e_{i,i+1}$ be the edge between $v_{i}$ and $v_{i+1}$ (where the indices are taken modulo $k$). Denote by $Z_{\pE}(\sF)$ the fiber product
\begin{equation}
\label{eq:limitExposed}
    Z_{\pE}(\sF) = Z(v_{1}) \times_{Z(e_{1,2})} Z(v_{2}) \times_{Z(e_{2,3})} \cdots \times_{Z(e_{k-\ell-2,k-\ell-1})} Z(v_{k-\ell-1})
\end{equation}
and let $\varphi_{\sF}:Z_{\pE}(\sF) \to Z(e_{1,k}) \times_{Z(F)} Z(e_{k-\ell-1,k-\ell})$ be the morphism induced by $Z(v_1 \to e_{1,k})$ and $Z(v_{k-\ell-1} \to e_{k-\ell-1,k-\ell})$.

\begin{figure}
    \centering
    \includegraphics[height=5cm]{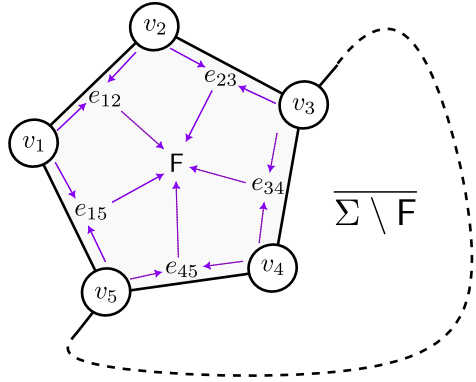}
    \caption{A fin and the corresponding quiver. The vertices $v_1$, $v_2$ and edges $e_{15}$, $e_{12}$, $e_{23}$ are exposed}
    \label{fig:attachFin}
\end{figure}

\begin{proposition}
\label{prop:addfin}
Suppose $Z:\Sigma \to \C\operatorname{-sch}$ is a diagram,  $\sF$ is a fin of $\Sigma$, and $\Sigma' = \overline{\Sigma\setminus \sF}$.  Then we have a pullback diagram
\begin{equation*}
\begin{tikzcd}
\varprojlim_{\Sigma} Z \arrow[r] \arrow[d] 
& Z_{\pE}(\sF)  \arrow[d, "\varphi_{\mathsf{F}}"] \\
\varprojlim_{\Sigma'} Z \arrow[r]
& Z(e_{1,k}) \times_{Z(F)} Z(e_{k-\ell-1,k-\ell})
\end{tikzcd}
\end{equation*}
\end{proposition}

The proof amounts to showing that $\varprojlim_{\Sigma}Z$ satisfies the universal property of fiber products. It is similar to \cite[Proposition~A.5]{CoreyGrassmannians}, so we omit the proof.

Given a collection of fins $\mathfrak{F}$, let $\Sigma(\mathfrak{F}) \subset \Sigma$ be the subcomplex obtained by removing, for each $\sF\in \mathfrak{F}$, the relatively open cell $\sF^{\circ}$ and the exposed vertices and edges of $\sF$. 

\begin{proposition}
\label{prop:limitFins}
If $\varprojlim_{\Sigma(\mathfrak{F})} Z$ is smooth and irreducible and $\varphi_{\sF}$ is an SDC-morphism for each $\sF \in \mathfrak{F}$, then $\varprojlim_{\Sigma} Z$ is smooth and irreducible of dimension
\small{\begin{equation}
\label{eq:limitFin}
    \dim \varprojlim_{\Sigma} Z = \dim \varprojlim_{\Sigma(\mathfrak{F})} Z + \sum_{\sF \in \mathfrak{F}} \left( \dim Z(\sF) - \sum_{e\in \pE_{\edge}(\sF)} \dim Z(e) + \sum_{v\in \pE_{\vertices}(\sF)} \dim Z(v)   \right)
\end{equation}}
\end{proposition}

\begin{proof}
This follows from Propositions \ref{prop:smoothIrreducibleFiberProduct} and \ref{prop:addfin} applied to each fin $\sF \in \mathfrak{F}$. 
\end{proof}

\section{Computations on thin Schubert cells}
\label{sec:computationsTSC}

By the Mn\"ev universality theorem (see \cite{Lafforgue2003,VakilLee}), the collection of schemes $\Gr(\sQ)$, as $\sQ$ runs through the $\C$-realizable $(3,n)$-matroids, satisfies Murphy's law in the sense of Vakil \cite{Vakil}. Nevertheless, for small values of $n$, many such $\Gr(\sQ)$ are smooth and irreducible; in fact $\Gr(\sQ)$ is smooth and irreducible for all $(3,n)$-matroid with $n\leq 7$ \cite[Proposition~4.2]{CoreyGrassmannians}. In this section, we show that $\Gr(\sQ)$ are smooth for all $\C$-realizable $(3,8)$--matroids, but there is one $\Gr(\sQ)$ that has 2 connected components.

\subsection{Affine coordinates for thin Schubert cells} 
\label{sec:affineCoordinates}
We recall the affine coordinate ring construction of $\Gr(\sQ)$, following \cite[Construction~2.2]{CoreyGrassmannians}.
Given any $r\times n$ (for $r\leq n$) matrix $A$, denote by $\col_{i}(A)$ the $i$-th column of $A$. With $\lambda = \{i_1<\cdots<i_s\}$ define the $r\times s$ submatrix $A_{\lambda} = [\col_{i_1} A, \ldots, \col_{i_s} A] $. 

Define $B$ to be the polynomial ring 
\begin{equation*}
    B = \C[ x_{ij} \, : \, i\in [r], \; j \in [n-r] ]
\end{equation*}
Given $\mu = \{j_1<\cdots<j_r\} \in \binom{[n]}{r}$, let $A^{(\mu)}$ be the $B$-valued matrix such that $(A^{(\mu)})_{\mu}$ is the $r\times r$ identity matrix, and $(A^{(\mu)})_{[n]\setminus \mu} $ is the $r\times(n-r)$ matrix $[x_{ij}]$. Of particular importance is the case $\mu = [r]$, where we write  $A := A^{(\mu)}$, which is
\begin{equation}
\label{eq:matrixA}
    A = \begin{bmatrix}
    1 & 0  & \cdots & 0 & x_{11} & x_{12} & \cdots & x_{1,n-r} \\
    0 & 1  & \cdots & 0 & x_{21} & x_{22} & \cdots & x_{2,n-r} \\
    \cdots & \cdots & \cdots & \cdots & \cdots & \cdots & \cdots & \cdots\\ 
    0 & 0 & \cdots & 1 & x_{r1} & x_{r2} & \cdots & x_{r,n-r} 
    \end{bmatrix}
\end{equation}

Let $\sQ$ be an $(r,n)$--matroid, $\mu$ a basis of $\sQ$, and $\sa_{\mu}:[n-r] \to [n]\setminus \mu$ the unique order-preserving bijection.  When $\mu = [r]$, this function is $\sa_{\mu}(j) = r + j$. Let
\begin{equation*}
    B_{\sQ}^{\mu} = \C[x_{ij} \, : \, \mu \Delta\{i, \sa_{\mu}(j)\} \in \sQ ]
\end{equation*}
where $\Delta$ denotes the symmetric difference, and define the quotient ring map
\begin{equation*}
    \pi_{\sQ}: B \to B / \langle x_{ij} \, : \, \mu \Delta\{i, \sa_{\mu}(j)\} \in \textstyle{\binom{[n]}{r}} \setminus  \sQ   \rangle  \cong B_{\sQ}^{\mu}.
\end{equation*}
Define the ideal $I_{\sQ}^{\mu}$ and multiplicative semigroup $S_{\sQ}^{\mu}$ by
\begin{equation}
\label{eq:IQSQ}
    I_{\sQ}^{\mu} = \langle \pi_{\sQ}(\det A^{(\mu)}_{\lambda}) \, : \, \lambda \in \textstyle{\binom{[n]}{r}} \setminus \sQ \rangle \hspace{20pt}  S_{\sQ}^{\mu} = \langle \pi_{\sQ}(\det A^{(\mu)}_{\lambda}) \, : \, \lambda \in \sQ \rangle_{\semigp}.  
\end{equation}
The affine coordinate ring of $\Gr(\sQ)$ is isomorphic to
\begin{equation*}
    R_{\sQ}^{\mu} := (S_{\sQ}^{\mu})^{-1} B_{\sQ}^{\mu} / I_{\sQ}^{\mu}.
\end{equation*}
To simplify notation, when $\mu = [r]$, we write $B_{\sQ} = B_{\sQ}^{\mu}$, $I_{\sQ} = I_{\sQ}^{\mu}$, $S_{\sQ} = S_{\sQ}^{\mu}$, and $R_{\sQ} = R_{\sQ}^{\mu}$. 

Given $(r,n)$-matroids $\sP,\sQ$ with  $\sP\leq \sQ$ in the face order (see \S\ref{sec:Grassmannian}) and $\mu$ a basis of $\sP$ (and therefore a basis of $\sQ$), the morphism $\varphi_{\sQ,\sP}: \Gr(\sQ) \to \Gr(\sP)$ is defined by the ring map \cite[Proposition~3.2]{CoreyGrassmannians}
\begin{equation}
\label{eq:TSCMapRings}
    (\varphi_{\sQ,\sP})^{\#}: R_{\sP}^{\mu} \to R_{\sQ}^{\mu} \hspace{20pt} x_{ij} \to x_{ij}.
\end{equation}

\begin{remark}\label{re:simplifycoord}
There is a further simplification one can make to compute the coordinate ring of $\Gr(\sQ)$. We describe this in the case when $\sQ$ is \textit{connected}.  Recall from the introduction that the diagonal torus $H\subset \PGL(n)$ acts on $\Gr(r,n)$. For any $\C$-realizable $(r,n)$-matroid, this restricts to an action $H\curvearrowright \Gr(\sQ)$. Since $\sQ$ is connected, this action is free. Let $X(\sQ) := \Gr(\sQ) / H$. As $\Gr(\sQ)$ is affine, we have $\Gr(\sQ) \cong X(\sQ) \times H$. 

Suppose $\lambda$ is a $(r+1)$-element subset of $[n]$ such that each $r$-element subset of $\lambda$ is a basis; such a $\lambda$ exists since $\sQ$ is connected. Applying a suitable permutation, we may assume that $\lambda = \{1,\ldots,r+1\}$. Let $A'$ be the matrix obtained from the matrix $A$ in Formula \eqref{eq:matrixA} by turning the $(r+1)$-st column into the column $[1, \cdots, 1]^T$. Consider the polynomial ring  
\begin{equation*}
    B' = \C[x_{12},\ldots,x_{r2}] \otimes_{\C} \cdots \otimes_{\C}  \C[x_{1,n-r},\ldots,x_{r,n-r}]
\end{equation*}
with the $\Z_{\geq 0}^{n-r-1}$-grading induced by the $n-r-1$ tensor-components. Define $\overline{I}_{\sQ}$ and $\overline{S}_{\sQ}$ in a way analogous to Formula \eqref{eq:IQSQ}, with $A'$ and $B'$ in place of $A$ and $B$, respectively. Then $X(\sQ)$ is the locally-closed subscheme of the multiprojective space $(\P^{r})^{n-r-1}$ cut out by the multihomogeneous ideal $\overline{I}_{\sQ}$ and localized at $\overline{S}_{\sQ}$. See \cite{BrennerSchroer} for scheme-theoretic treatments of multiprojective schemes. For each $2 \leq j \leq n-r$, there is an $i\in [r]$ such that $[r] \Delta \{i,r+j\}$ is a basis of $\sQ$, and so we may set $x_{ij} = 1$. Dehomogenizing in this way produces the affine coordinate ring of $X(\sQ)$. 

\end{remark}
\subsection{The (3,8) case}
\label{sec:38mat}
Let $\spMat$ be the matroid from Figure \ref{fig:spMatroidIntro}, i.e., it is the $(3,8)$-matroid whose bases are the elements of $\binom{[8]}{3}$ other than 
\begin{equation}
\label{eq:spMat}
    126, 145, 178, 235, 248, 347, 368, 567.
\end{equation}

\begin{proposition}
\label{prop:spMat2Components}
The thin Schubert cell $\Gr(\spMat)$ is a smooth affine scheme with two connected components of dimension $7$.  
\end{proposition}

\begin{proof}
As $X(\spMat) \cong \Gr(\spMat) \times H$ and $\dim H = 7$, it suffices to show that $X(\spMat)$ consists of two reduced points. Any $r\times n$ matrix representing $\spMat$ is in the $(\GL(r) \times H)$-orbit of 
\begin{equation}
\label{eq:matrixMatroidsp}
A^{\prime}=\begin{bmatrix}
1 & 0 &0 &1 & 0 & 1 &1 &1\\ 
0 & 1 &0 &1 & 1 & a &1 &a\\ 
0 & 0 &1 &1 & 1 & 0 & 1-a &1\\ 
\end{bmatrix}
\end{equation}
where $a$ is one of the 2 solutions of  $x^2-x+1=0$. This amounts to the fact that the affine coordinate ring of $X(\sQ)$ is isomorphic to 
\begin{equation*}
    \C[a] / \langle a^2 - a + 1 \rangle \cong \C \times \C. 
\end{equation*}
which may be verified using Remark \ref{re:simplifycoord}. 
\end{proof}

\begin{proposition}
\label{prop:smoothIrredTSC}
If $\sQ$ is a $\C$-realizable $(3,8)$-matroid other than $\spMat$, then $\Gr(\sQ)$ is smooth and irreducible. 
\end{proposition}

\begin{proof}
The $(3,8)$-matroids have been enumerated by \cite{MatsumotoMoriyamaImaiBremner}. Up to $\Sn{8}$-symmetry, there are 325 such matroids, 68 of which are simple. Of the simple matroids, all but the ones isomorphic to $\spMat$ have the following property: there exists an $i \in [8]$ that is contained in 2 or fewer rank-2 cyclic flats. In fact, every element of $[8]$ is contained in exactly 3 rank-2 cyclic flats of $\spMat$. Thus the proposition follows from \cite[4.1-2]{CoreyGrassmannians}. 
\end{proof}

\subsection{Maps between thin Schubert cells} 
  Given a connected $(r,n)$-matroid $\sQ$ and $\sP \leq \sQ$ (i.e., $\Delta(\sP)$ is a face of $\Delta(\sQ)$), we say that $\sP$ is \textit{internal} if $\Delta(\sP)$ is contained in the interior of $\Delta(r,n)$.   If $\sP \leq \sQ$ are $(3,n)$-matroids for $n\leq 6$, or if $\sP\leq \sQ$ are $(3,7)$-matroids such that $\sQ$ is connected and $\sP$ is internal, then $\Gr(\sQ) \to \Gr(\sP)$ is an SDC-morphism \cite[Propositions~5.4,5.5]{CoreyGrassmannians}. In general, the morphisms $\varphi_{\sQ,\sP}: \Gr(\sQ) \to \Gr(\sP)$ are not expected to be smooth or dominant when $\sP \leq \sQ$, and we see this when  $n=8$.
  
  One such example is depicted in Figure \ref{fig:nondominantMap}. In the matroid on the left, the points labeled $1,2,3,4$ all lie on a $\P^{1}$, and the cross-ratio $(1,3;2,4)$ equals $-1$, see \cite[\S 5.2]{GelfandGoreskyMacPhersonSerganova}.  However, there is no such constraint for the matroid on the right. So the morphism $\varphi_{\sQ,\sP}$ is not dominant, nor is it smooth (as it is not open).  Nevertheless, enough of the morphisms between $(3,8)$-matroid strata are smooth and dominant with connected fibers to allow us to prove Theorem \ref{thm:schoenIntro}. We record some observations to simplify our analysis in the following 2 sections.

\begin{figure}[tbh]
    \centering
    \includegraphics[width=0.8\textwidth]{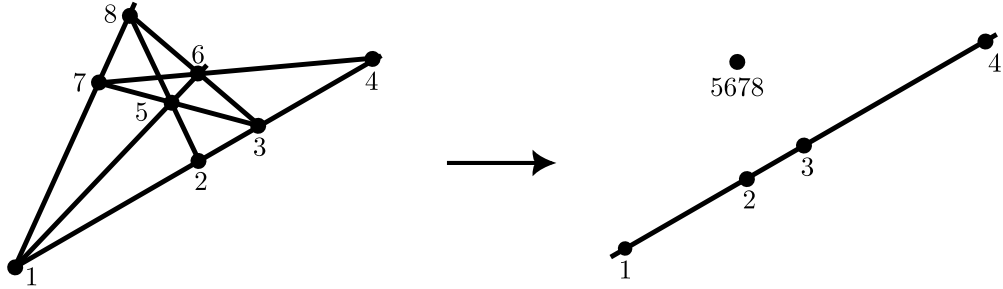}
    \caption{An example of a non-smooth, non-dominant morphism of thin Schubert cells}
    \label{fig:nondominantMap}
\end{figure}

Let $\sQ$ be a loopless $(r,n)$--matroid whose rank-1 flats are $\eta_1,\ldots,\eta_k$. Choose $a_1 \in \eta_1,\; a_2\in \eta_2, \ldots, a_k\in \eta_k$ and set $S = \{a_1,\ldots,a_k\}$. Then $\sQ|S$ is a simple matroid, and is called a \textit{simplification} of $\sQ$. Such a set $S$ is called a \textit{simplifying set} of $\sQ$. Given a vector $\sv \in N_{\R}$, denote by $\face_{\sv}\sQ$ the matroid
\begin{equation*}
    \face_{\sv}\sQ = \{ \xi \in \sQ \, : \, \langle \se_{\xi}^{*}, \sv \rangle \leq  \langle \se_{\xi'}^{*}, \sv \rangle \text{ for all } \xi'\in \sQ  \};
\end{equation*}
this is the matroid of the face of $\Delta(\sQ)\subset M_{\R}$ normal to $\sv \in N_{\R}$.
Now suppose $\sQ$ is connected. A subset  $\lambda \subset [n]$ is a \textit{nondegenerate subset} of $\sQ$ if $\sQ|\lambda$ and $\sQ/\lambda$ are connected.   Recall from \cite[\S~2.5]{GelfandSerganova} that $\face_{-\epsilon_{\lambda}}\sQ \cong \sQ|\lambda \times \sQ/\lambda$ and $\lambda \mapsto \Delta(\face_{-\epsilon_{\lambda}}\sQ)$ defines a bijection between the nondegenerate subsets of $\sQ$ and the facets of $\Delta(\sQ)$. 

Now suppose $\sQ$ has rank $3$. A subset $\lambda\subset [n]$ is a \textit{line} if $\lambda$ is a rank-2 flat and $\lambda\cap S$ is a cyclic flat of $\sQ|S$; equivalently, $\lambda$ is a rank-2 flat such that $|\lambda \cap S| \geq 3$. If $\sQ$ is simple, then a line is simply a rank-2 cyclic flat, but in general being a line is a slightly stronger condition than being a rank-2 cyclic flat. 

\begin{proposition}
\label{prop:oneInternalFacet}
Suppose $\sQ$ is a connected $(3,n)$--matroid with 2 or fewer lines.  Then $\Gr(\sQ)$ is smooth and irreducible, and $\varphi_{\sQ,\sP}: \Gr(\sQ) \to \Gr(\sP)$ is an SDC-morphism for any internal $\sP\leq \sQ$. 
\end{proposition}

\noindent The analog of the second part for matroids $\sQ$ with 3 or more lines is false by \cite[Example~8.4]{CoreyGrassmannians}.

\begin{proof}
That $\Gr(\sQ)$ is smooth and irreducible follows from \cite[4.1-2]{CoreyGrassmannians}. Suppose $\sP\leq \sQ$ is internal. Given a simplifying set $S$ of $\sQ$, we have that  $\sP|S \leq \sQ|S$ and $\Gr(\sQ) \to \Gr(\sP)$ is an SDC-morphism provided $\Gr(\sQ|S) \to \Gr(\sP|S)$ is an SDC-morphism \cite[Lemma~C.5]{CoreyGrassmannians} (Appendix by Cueto). However,  $\sP|S \leq \sQ|S$ may no longer be internal. Therefore, it suffices to prove that 
$\varphi_{\sQ,\sP}$ is an SDC-morphism if $\sQ$ is simple with fewer than 2 lines, and $\sP\leq \sQ$ such that $\Delta(\sP)$ is a (not necessarily internal) facet of $\Delta(\sQ)$. 

By \cite[Proposition~5.1]{CoreyGrassmannians}, $\sP = \face_{-\epsilon_{\lambda}} \sQ$ where $|\lambda|\in\{1,n-1\}$ or $\lambda$ is a line.  If $\lambda$ is a line or $|\lambda| = n-1$, then $\varphi_{\sQ,\sP}$ is an SDC-morphism by \cite[4.1,5.3-4]{CoreyGrassmannians}. We are left to consider the case $|\lambda| = 1$, say $\lambda = \{a\}$. Suppose $\sQ$ has 2 lines $\lambda_1,\lambda_2$, which meet at a point. 

\textbf{Case 1}: $a\in \lambda_2 \setminus \lambda_1$. Say $a=3$, $\lambda_1 = \{1,2,4,\ldots,k+3\}$ and $\lambda_2 = \{1,3,k+4,\ldots,\ell+3\}$. Then $\Gr(\sQ)\to \Gr(\sP)$ is an SDC-morphism because
\begin{equation*}
    R_{\sQ} \cong S_{\sQ}^{-1}R_{\sP}[x_{3,k+1}^{\pm},\ldots,x_{3n}^{\pm}].
\end{equation*}

\textbf{Case 2}: $a\in \lambda_1 \cap \lambda_2$. We may assume $\lambda_1$ and $\lambda_2$ are as in the previous case, and that $a=1$. Then $\Gr(\sQ)\to \Gr(\sP)$ is an SDC-morphism because
\begin{equation*}
    R_{\sQ} = S_{\sQ}^{-1}R_{\sP}[x_{11}^{\pm},\ldots,x_{1n}^{\pm}]
\end{equation*}

\textbf{Case 3}: $a\notin \lambda_1 \cup \lambda_2$. Suppose $a=1$, $\lambda_1 = \{2,3,\ldots,k+3\}$, $\lambda_2 = \{3,k+4,\ldots,\ell+3\}$.
\begin{equation*}
    R_{\sQ} = S_{\sQ}^{-1}R_{\sP}[x_{1,k+1}^{\pm},\ldots,x_{1n}^{\pm}] / \langle x_{1c} x_{2\ell} - x_{2c} x_{1\ell} \, : \, k < c < \ell \rangle
\end{equation*}
As $x_{1c} \equiv x_{2c} x_{1\ell} / x_{2\ell}$ for $k<c<\ell$ in $R_{\sQ}$, we see that $R_{\sQ} \cong S_{\sQ}^{-1}R_{\sP}[x_{1,\ell}^{\pm},\ldots,x_{1n}^{\pm}]$, and so $\Gr(\sQ) \to \Gr(\sP)$ is an SDC-morphism.

Now suppose $\lambda_1 \cap \lambda_2 = \emptyset$.

\textbf{Case 1}: $a\in \lambda_1$. Say $a=3$, $\lambda_1 = \{3,4,\ldots,k+3\}$ and $\lambda_2 = \{1,2,k+4,\ldots,\ell+3\}$. Then
\begin{equation*}
    R_{\sQ} = S_{\sQ}^{-1}R_{\sP}[x_{3,k+1}^{\pm},\ldots,x_{3,n}^{\pm}] 
\end{equation*}

\textbf{Case 2}: $a \notin \lambda_1\cup \lambda_2$. Say $a=3$, $\lambda_1 = \{1,4,\ldots,k+3\}$ and $\lambda_2 = \{2,k+4,\ldots,\ell+3\}$. Then $\Gr(\sQ) \to \Gr(\sP)$ is an SDC-morphism since
\begin{align*}
 R_{\sQ} &= S_{\sQ}^{-1}R_{\sP}[x_{3,1}^{\pm},\ldots,x_{3,n}^{\pm}] \left/ \left\langle
 \begin{array}{ll}
    x_{2c} x_{3k} - x_{3c} x_{2k}  &  1 \leq c < k  \\
    x_{1c} x_{3\ell} - x_{3c} x_{1\ell} & k < c < \ell 
 \end{array} \right\rangle \right. \\
  &\cong S_{\sQ}^{-1}R_{\sP}[x_{31}^{\pm},x_{3k}^{\pm},x_{3\ell}^{\pm},x_{3,\ell+1}^{\pm},\ldots,x_{3n}^{\pm}]
\end{align*}
The cases where $\sQ$ has 0 or 1 lines are handled in a similar fashion (and in fact, are easier). 
\end{proof}

\begin{figure}[tbh!]
    \centering
    \includegraphics[width=0.8\textwidth]{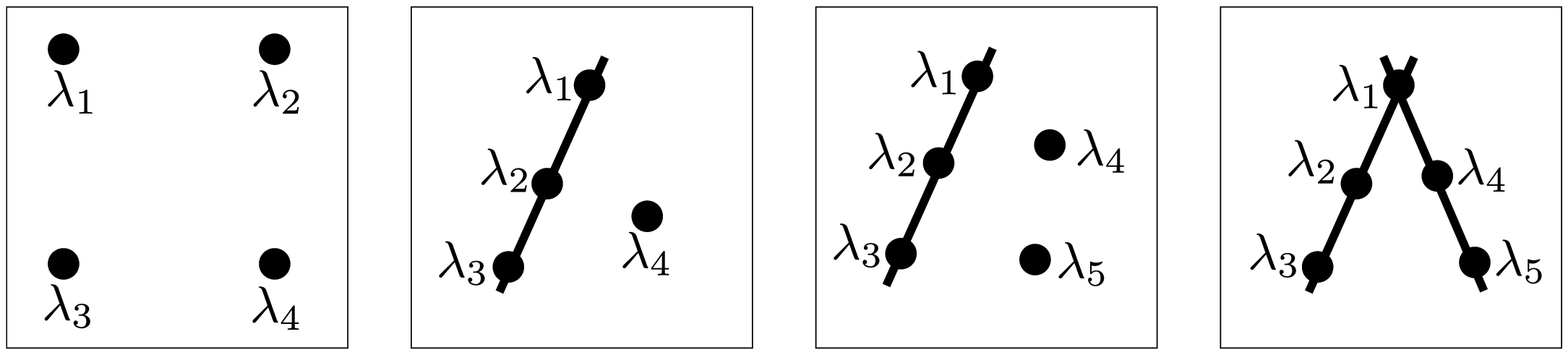}
    \caption{The matroids $\sU,\sU',\sV,\sW$}
    \label{fig:matroidsUVW}
\end{figure}

For many examples in this paper, we must isolate a few types of matroids. Given a partition of $[n]$ into $\lambda_1, \lambda_2, \lambda_3, \lambda_4$, let $\sU(\lambda_1, \lambda_2, \lambda_3, \lambda_4)$ be the matroid on the left in Figure \ref{fig:matroidsUVW}; i.e., this is the unique matroid whose rank-1 flats are the $\lambda_i$'s and whose simplification is the uniform $(3,4)$--matroid. Define in a similar fashion $\sU'(\lambda_1,\lambda_2,\lambda_3;\lambda_4)$, $\sV(\lambda_1, \lambda_2, \lambda_3 ; \lambda_4, \lambda_5)$ and $\sW(\lambda_1 ; \lambda_2, \lambda_3 ; \lambda_4, \lambda_5)$  the remaining matroids in Figure \ref{fig:matroidsUVW}. When clear from the context, we simply write $\sU,\sU',\sV,\sW$ for these matroids.  We record the following computation. 

\begin{proposition}
\label{prop:TSCUVW}
The thin Schubert cells of the matroids $\sU$, $\sU'$, $\sV$, and $\sW$ are smooth and irreducible. Their dimensions are
{\small
\begin{align*}
\begin{array}{ll}
      \dim \Gr(\sU) = |\lambda_1| + |\lambda_2| + |\lambda_3| + |\lambda_4| - 1 &
    \dim \Gr(\sU') = |\lambda_1| + |\lambda_2| + |\lambda_3| + |\lambda_4| - 2 \\
     \dim \Gr(\sV) =  |\lambda_1| + |\lambda_2| + |\lambda_3| + |\lambda_4| + |\lambda_5| &
    \dim \Gr(\sW) = |\lambda_1| + |\lambda_2| + |\lambda_3| + |\lambda_4| + |\lambda_5| - 1 
\end{array}
\end{align*}
}
Furthermore, for any pair $\sQ'\leq \sQ$ with $\sQ \cong \sU,\sV,\sW$, the map $\varphi_{\sQ,\sQ'}$ is an SDC-morphism. 
\end{proposition}

\begin{proof}
The statements on smoothness and irreducibility of the thin Schubert cells, and that the maps are SDC-morphisms, follow from Proposition \ref{prop:oneInternalFacet}. 
To compute their dimensions, we can apply \cite[Lemmas~C.2,C.5]{CoreyGrassmannians} (Appendix by Cueto) to reduce to the simple-matroid case. For example, consider the matroid $\sQ = \sW(1;2,4;3,5)$. The ring $R_{\sQ}$ is
\begin{equation*}
    R_{\sQ} = \C[x_{11}^{\pm},x_{21}^{\pm},x_{21}^{\pm},x_{23}^{\pm}]
\end{equation*}
and therefore $\dim \Gr(\sQ) = 4$. The remaining cases may be handled in a similar fashion.
\end{proof}

\subsection{$B$-maximality}
Given an $(r,n)$-matroid $\sQ$ and a basis $\mu \in \sQ$, define
\begin{equation}
    d(\sQ,\mu) = \left| \{ \lambda \in \sQ \, : \, |\lambda \, \Delta \, \mu| = 2 \} \right|
\end{equation}
Observe that $d(\sQ,\mu)$ is the Krull dimension of $B_{\sQ}^{\mu}$, and so $d(\sQ,\mu) \geq \dim \Gr(\sQ)$.  Intuitively, a smaller $d(\sQ,\mu)$ produces a simpler ideal $I_{\sQ}^{\mu}$.  We say that $\sQ$ is $(B,\mu)$-\textit{maximal} if $\dim \Gr(\sQ) = d(\sQ,\mu)$, and $\sQ$ is $B$-\textit{maximal} if there is a $\mu\in \sQ$ such that $\sQ$ is $(B,\mu)$-maximal.

\begin{proposition}
\label{prop:conditionIdealTrivial}
Suppose $\sQ$ is a $(B,\mu)$-maximal matroid for some $\mu\in \sQ$.   
\begin{enumerate}
    \item The thin Schubert cell $\Gr(\sQ)$ is an open subscheme of $(\GG_m)^{d(\sQ,\mu)}$, and therefore smooth and irreducible. 
    \item If $\sP \leq \sQ$, then $\sP$ is an open subscheme of $(\GG_m)^{d(\sP,\mu)}$, and $\varphi_{\sQ,\sP}: \Gr(\sQ) \to \Gr(\sP)$ is an SDC-morphism.  
\end{enumerate}
\end{proposition}

\begin{proof}
Without loss of generality, assume that $\mu = [r]$, in particular $\sP$ and $\sQ$ have $[r]$ as a basis. The equality $\dim \Gr(\sQ) = d(\sQ, \mu)$ implies that $I_{\sQ} = \langle 0 \rangle$; this proves (1). 

Given the nature of the map $(\varphi_{\sQ,\sP})^{\#}:R_{\sQ} \to R_{\sP}$ as described by Formula \ref{eq:TSCMapRings} we also have that $I_{\sP} = \langle 0\rangle$. With $a=d(\sP,\mu)$ and $b=d(\sQ,\mu)$, we have that $\Gr(\sP)$ and $\Gr(\sQ)$  are open subschemes of $\GG_m^{a}$ and $\GG_m^{b}$, respectively. The morphism $\varphi_{\sQ,\sP}:\Gr(\sQ) \to \Gr(\sP)$ is induced by a coordinate projection $\GG_m^a \to \GG_m^b$, and therefore it is an SDC-morphism. This completes the proof of (2). 
\end{proof}

Suppose $\sw \in N(\sQ)$ such that $\pQ(\sw)$ is matroidal, $\Sigma\subset \TS(\sw)$ a connected subcomplex (in most applications, $\Sigma$ is the subcomplex $\Sigma_{\sL}$ obtained by removing all leaf vertices and edges from $\TS(\sw)$), and  $\sF$ is a fin of $\Sigma$. Assume that the vertices and edges of $\sF$ are labeled as in \S \ref{sec:removeFins}, and let $\sQ_{i}$, $\sQ_{i,i+1}$, $\sQ_{\sF}$ be the matroids corresponding to the vertex $v_i$, edge $e_{i,i+1}$, and fin $\sF$, respectively. Denote by $\Gr_{\pE}(\sF)$ the fiber product from Formula \eqref{eq:limitExposed} applied to $Z = \Gr$. The morphisms $\Gr(\sQ_1) \to \Gr(\sQ_{1,k})$ and $\Gr(\sQ_k) \to \Gr(\sQ_{k-\ell,k-\ell+1})$ induce a morphism
\begin{equation*}
    \varphi_{\sF}: \Gr_{\pE}(\sF) \to \Gr(\sQ_{1,k}) \times_{\Gr(\sQ_{\sF})}  \Gr(\sQ_{k-\ell,k-\ell+1}).
\end{equation*}
We say that the fin $\sF$ is $B$-\textit{maximal} if there is a $\mu \in \sQ_{\sF}$ such that  $d(\sQ_i,\mu) = \dim \Gr(\sQ_i)$ for $i=1,\ldots,k-\ell$. 

\begin{proposition}
\label{prop:finSDC}
If $\sF$ is $B$-maximal, then $\varphi_{\sF}$ is an SDC-morphism.
\end{proposition}

\begin{proof}
    Without loss of generality, assume that $\mu = [r]$; in particular each $\sQ_i$, $\sQ_{i,i+1}$, and $\sQ_F$ has $[r]$ as a basis. By Proposition \ref{prop:conditionIdealTrivial}, we have that
    \begin{equation*}
         \Gr(\sQ_i) \subset (\GG_m)^{d(\sQ_i,\mu)} \hspace{20pt} \Gr(\sQ_{i,i+1}) \subset (\GG_m)^{d(\sQ_{i,i+1},\mu)} \hspace{20pt} \Gr(\sQ_{\sF}) \subset (\GG_m)^{d(\sQ_{\sF},\mu)} 
    \end{equation*}
    as open subschemes, and the morphisms
    \begin{equation*}
        \Gr(\sQ_{i}) \to \Gr(\sQ_{i,i+1}) \hspace{20pt} \Gr(\sQ_{1,k}) \to \Gr(\sQ_{\sF}) \hspace{20pt} \Gr(\sQ_{k-\ell,k-\ell+1}) \to \Gr(\sQ_{\sF}) 
    \end{equation*}
    are induced by coordinate projections of tori.  This implies that the schemes $\Gr_{\pE}(\sF)$ and  $\Gr(\sQ_{1,k}) \times_{\Gr(\sQ_{\sF})}  \Gr(\sQ_{k-\ell,k-\ell+1})$ may be realized as open subschemes of tori, and the morphism $\varphi_{\sF}$ is induced a coordinate projection of tori. Therefore,  $\varphi_{\sF}$ is an SDC-morphism. 
\end{proof}

\section{A reducible initial degeneration}
\label{sec:reducibleInitDeg}

Recall from the introduction and \S\ref{sec:38mat} that $\spMat$ is the $(3,8)$-matroid with nonbases as in \eqref{eq:spMat}. Let $\spw$ be the co-rank vector of this matroid. The cone of $\pS_{\trop}(3,8)$, the secondary fan structure of $\TGr_0(3,8)$ described in \ref{sec:Grassmannian}, containing $\spw$ in its relative interior is
\begin{equation*}
    \spCone = \R_{\geq 0} \langle \se_{126}, \se_{145}, \se_{178}, \se_{235}, \se_{248}, \se_{347}, \se_{368}, \se_{567} \rangle + L_{\R}. 
\end{equation*}
The tight-span $\TS(\spw)$ is a star-shaped tree. The matroid of the center node is $\spMat$. There are 8 leaves, one corresponding to each nonbasis of $\spMat$. The matroids of the leaf vertices and adjacent edges are
\begin{equation*}
    \sQ(ijk)  = \sU(i,j,k, [8]\setminus ijk), \hspace{10pt} \text{and} \hspace{10pt} \sQ'(ijk)  = \sU'(i,j,k; [8]\setminus ijk),
\end{equation*}
respectively, where $\sU$ and $\sU'$ are the left 2 matroids in Figure \ref{fig:matroidsUVW}, and $ijk$ run through the nonbases of $\spMat$.

\begin{proposition}
\label{prop:specialLimit}
Let $\pC$ be any cone of $\pS_{\trop}(3,8)$ in the $\Sn{8}$-orbit of $\spCone$, and $\sw$ in the relative interior of $\pC$. The inverse limit $\Gr(\sw)$ is smooth, $15$-dimensional, and has 2 connected components. 
\end{proposition}

\begin{proof}
It suffices to consider only $\sw = \spw$.  With
\begin{equation*}
    Y_{\Leaf} =\prod_{ijk} \sQ(ijk) \hspace{20pt} \text{and} \hspace{20pt} Y_{\Leaf}' = \prod_{ijk} \sQ'(ijk). 
\end{equation*}
we have 
\begin{equation*}
    \Gr(\sw) \cong \Gr(\spMat) \times_{Y_{\Leaf}'} Y_{\Leaf} 
\end{equation*}
As the morphism $Y_{\Leaf} \to Y_{\Leaf}'$ is smooth and surjective with connected fibers, its pullback  $\varphi:\Gr(\sw) \to \Gr(\spMat)$ is also smooth and surjective with connected fibers. The thin Schubert cell $\Gr(\spMat)$ is smooth and 7-dimensional with 2 connected components by Proposition \ref{prop:spMat2Components}. Therefore, $\Gr(\sw)$ is smooth and 15-dimensional with 2 connected components by Proposition \ref{prop:smoothIrreducibleFiberProduct}. 
\end{proof}

\noindent We  prove Theorem \ref{thm:2ConnectedComponentsInDegIntro} in the following form. 
\begin{theorem}
\label{thm:2connectedComponents}
Let $\pC$ be any cone of $\pS_{\trop}(3,8)$ in the $\Sn{8}$-orbit of $\spCone$, and $\sw$ in the relative interior of $\pC$. The initial degeneration $\init_{\sw} \Gr_0(3,8)$ is smooth and has two connected components.
\end{theorem}

\begin{proof}
It suffices to consider the case $\sw = \spw$. 
Being the limit of a flat degeneration of $\Gr_0(3,8)$, the initial degeneration  $\init_{\sw} \Gr_0(3,8)$ is 15-dimensional. As $\init_{\sw}\Gr_0(3,8)$ is a smooth and 15-dimensional affine scheme with 2 connected components, the closed immersion from Theorem \ref{thm:closedImmersion} is an isomorphism if its image meets the two components of $\Gr(\sw)$. 
Denote by $\psi:\init_{\sw}\Gr_0(3,8) \to X(\spMat)$ the composition
\begin{equation*}
 \init_{\sw}\Gr_0(3,8) \xrightarrow{\varphi_{\sw}}  \Gr(\sw) \to \Gr(\spMat) \to X(\spMat)
\end{equation*}
where the middle morphism is $\varphi$ from the proof of Proposition \ref{prop:specialLimit} and the right morphism is the quotient by the diagonal torus $H\subset \PGL(8)$; both morphisms are surjective with connected fibers. Consider the $\C(t)$-valued matrix
\begin{equation*}
    A_t = \begin{bmatrix}
     1 & 0 & 0 & 1 & t & 1+4t & 1-t & 1-t^2 \\
     0 & 1 & 0 & 1+t & 1+t & a+t & 1-2t & a-t \\
     0 & 0 & 1 & 1+2t & 1+3t & 2t & 1-a + t & 1+t^2
\end{bmatrix}
\end{equation*}
where $a\in \C$ is a solution of $x^2-x+1=0$. Denote by $\vec{p}(A_{t})$ the vector of homogeneous Pl\"ucker coordinates of $A_{t}$. The tropicalization  (i.e., coordinatewise valuation with respect to the $t$-adic valuation on $\C(t)$) of  $\vec{p}(A_t)$ is $\sw$, so its exploded tropicalization  $\vec{x}(A_{t}) = \mathfrak{Trop}(\vec{p}(A_t))$ (in the sense of \cite{Payne}) lies in $\init_{\sw}\Gr_0(3,8)$. By \cite[Remark~3.6]{CoreyGrassmannians} the morphism $\psi$ takes $\vec{x}(A_{t})$ to $\vec{p}(A)$, where $A$ is the matrix from Formula \eqref{eq:matrixMatroidsp}. Switching between the two solutions $a$ of $x^2-x+1 =0$, this shows that $\psi$ is surjective. As the composition $\Gr(\sw) \to X(\spMat)$ is continuous and surjective and $\varphi_{\sw}$ is continuous, the image of $\varphi_{\sw}$ meets the two connected components of $\Gr(\sw)$, as required. 
\end{proof}

\section{Proofs of Theorem \ref{thm:schoenIntro} and \ref{thm:connectedInitDegIntro}}
\label{sec:proofThms1And3}

In this section we prove Theorems  \ref{thm:schoenIntro} and \ref{thm:connectedInitDegIntro}. We do so by showing that, except for $\sw$ contained in the relative interior of a cone in the $\Sn{8}$--orbit of $\spCone$, the inverse limits $\Gr(\sw)$ of $(3,8)$-matroid strata are smooth and irreducible of dimension $15$, and therefore so are $\init_{\sw}\Gr_0(3,8)$ by Corollary \ref{cor:limit2Init}.  We begin by describing how to compute the coordinate ring of an inverse limit of thin Schubert cells using the affine coordinates from \S \ref{sec:affineCoordinates}.

\subsection{Coordinate rings for finite inverse limits of thin Schubert cells}

First, consider the more general setup.  Let $\Delta\subset M_{\R}$ be any lattice polytope with vertex-set $\sQ$. For $\sw \in N(\sQ)_{\R}$, let $\pQ(\sw)$ be the corresponding regular subdivision of $\Delta$, and $\TS(\sw)$ its tight-span.  Given a cell $\sC$ of $\TS(\sw)$, denote by $\Delta_{\sC}$ the corresponding cell of $\pQ(\sw)$. Let $\Sigma \subset \TS(\sw)$ be a connected subcomplex, and $\Delta \subset \pQ(\sw)$ its dual (i.e., the union of cells $\Delta_{\sC}$ for $\sC \subset \Sigma$). We say that $\Sigma$ is \textit{vertex-intersecting} if
\begin{equation*}
    \bigcap_{v \in V(\Sigma)} \Delta_{v} \neq \emptyset
\end{equation*}
and $\Sigma$ is \textit{vertex-connecting} for each vertex $x $ of $\Delta$, the subcomplex of $\Sigma$ formed by the cells $\sC$ with $x \in \Delta_{\sC}$ is connected. 

\begin{proposition}
\label{prop:vertexConnectingLeaf}
Suppose $(v,e)$ is a leaf pair of $\Sigma$, and $\Sigma'$ is the subcomplex obtained by removing $v$ and $e$. If  $\Sigma$ is vertex-connecting  then  $\Sigma'$ is vertex-connecting.  
\end{proposition}

\begin{proof}
This follows from the fact that there is a hyperplane in $M_{\R}$ that separates $\Delta_{v}$ from the rest of $\Delta$ along $\Delta_e$. 
\end{proof}

\begin{proposition}
\label{prop:vertexConnectingFin}
Suppose $\sF$ is a fin of $\Sigma$ whose connecting path has length 1, and $\Sigma'$ is the subcomplex obtained by removing $\sF^{\circ}$ and all of its exposed vertices and edges. If $\Sigma$ is vertex-connecting then $\Sigma'$ is vertex-connecting. 
\end{proposition}

\begin{proof}
Suppose $\sC$ and $\sC'$ are two cells of $\Sigma'$ such that $x \in \Delta_{\sC}\cap \Delta_{\sC'}$. As $\Sigma$ is vertex-connecting, there is a sequence $\sC_* = (\sC_1, \sC_2,\ldots,\sC_{k-1},\sC_k)$ of \textit{distinct} cells in $\Sigma$ such that \begin{itemize}
    \item[-] $\sC = \sC_1$ and $\sC' = \sC_k$,
    \item[-] $\sC_{i}$ is a face of $\sC_{i+1}$ (or vice versa), and 
    \item[-] $x\in \Delta_{\sC_i}$ for each $i=1,\ldots,k$.
\end{itemize}
If $\sC_*$ passes through $\sF$, then  $x\in \Delta_{\sF}$ and therefore $x \in \Delta_{e}$ where $e$ is the connecting edge. So one may modify $\sC_*$ so that it misses $\sF$, as well as its exposed vertices and edges. If $\sC_*$ does not pass through $\sF$ but does pass through an exposed vertex or edge, then it must pass through all exposed vertices and edges, as well as the vertices of the connecting edge $e$. This implies again that $x\in \Delta_{e}$, so we may modify $\sC_{*}$ to be a path entirely in $\Sigma'$. 
\end{proof}

Now suppose $\sQ$ is a matroid, $\sw \in \Dr(\sQ)$, and $\Sigma\subset \TS(\sw)$ is a connected subcomplex.  Given a cell $\sC$ of $\Sigma$, let $\sQ_{\sC}$ be its corresponding matroid.  Suppose $\Sigma$ is vertex-intersecting; this means that there is a basis common to the $\sQ_v$ for $v\in V(\Sigma)$. For simplicity, assume that the common basis is $[r]$.  Define
\begin{equation*}
    B_{\Sigma} = \C[x_{ij} \, : \, [r] \Delta \{i,r+j\} \in \sQ_{v} \text{ for some } v \in V(\Sigma)]
\end{equation*}
Next, define the ideal and multiplicative semigroup 
\begin{equation*}
    I_{\Sigma} = \sum_{v \in V(\Sigma)} I_{\sQ_v} \cdot B_{\Sigma} \hspace{20pt} \text{and} \hspace{20pt} S_{\Sigma} = \langle S_{\sQ_v} \, : \, v \in V(\Sigma) \rangle_{\semigp}; 
\end{equation*}
note that we may view each $S_{\sQ_v}$ as a subset of $B_{\Sigma}$ under the inclusion $B_{\sQ_v} \subset B_{\Sigma}$. Finally, set
\begin{equation}
    R_{\Sigma} = (S_{\Sigma})^{-1} B_{\Sigma} / I_{\Sigma}
\end{equation}

\begin{proposition}
\label{prop:coordRingSigma}
If $\Sigma$ is vertex-intersecting and vertex-connecting, then the coordinate ring of the inverse limit $\varprojlim_{\Sigma} \Gr$ is isomorphic to $R_{\Sigma}$. 
\end{proposition}

\begin{proof}
The proof is similar to that of \cite[Proposition~6.5]{CoreySpinor}, so we only provide a sketch. The coordinate ring of $\varprojlim_{\Sigma} \Gr$ is $\varinjlim_{\Sigma} R_{\sQ}$ (indeed, $\Spec$ is  right-adjoint to the global-sections functor, and hence takes direct limits to inverse limits). For each cell $\sC$ of $\Sigma$, the inclusion $B_{\sQ_{\sC}} \subset B_{\Sigma}$ induces a ring morphism $R_{\sQ_{\sC}} \to R_{\Sigma}$. By the universal property, these induce a ring map $\Phi:\varinjlim_{\Sigma} R_{\sQ} \to R_{\Sigma}$. Now let us define the inverse $\Psi: R_{\Sigma} \to \varinjlim_{\Sigma} R_{\sQ}$. Suppose $x_{ij} \in B_{\Sigma}$, and $\sC$ is a cell such that $x_{ij}\in B_{\sQ_{\sC}}$. Set $\Psi(x_{ij}) = \alpha_{\sC}(x_{ij})$ where $\alpha_{\sC}:R_{\sQ_{\sC}} \to \varinjlim_{\Sigma} R_{\sQ}$ is the structure map. Suppose $x_{ij}$ also lies in $B_{\sQ_{\sC'}}$; as $\Sigma$ is vertex-connecting, we may reduce to the case where $\sC'$ is a face of $\sC$. Then $\alpha_{\sC'}(x_{ij}) = (\varphi_{\sQ_{\sC},\sQ_{\sC'}})^{\#} \circ \alpha_{\sC} (x_{ij})$, and so $\Psi:B_{\Sigma} \to \varinjlim_{\Sigma} R_{\sQ}$ is well-defined. All that is left is to show that $\Psi$ passes to the localization and quotient $R_{\Sigma}$, which is a verification we leave to the reader. 
\end{proof}

\subsection{Combinatorial classification of diagrams}
The secondary fan structure of the tropical Grassmannian $\TGr_0(3,8)$ is computed in \cite{BendleBoehmRenSchroter}. It has a $7$-dimensional lineality space, and its $f$-vector is
{\small
\begin{equation*}
    f(\TGr_0(3,8)) = (f_{7},\ldots,f_{15}) \equiv (1, 12, 155, 1149, 5013, 12\,736, 18\,798, 14\,714, 4766) \mod \Sn{8}. 
\end{equation*}}

\noindent A \textit{combinatorial type} is a $\Sn{8}$--orbit of a cone in $\pS_{\trop}(3,8)$, and we typically record a combinatorial type by a vector $\sw$ in the relative interior of a cone in an $\Sn{8}$-orbit. There are $57\, 344$ combinatorial types of $\pS_{\trop}(3,8)$. Consider the following conditions (see \S \ref{sec:removeLeaf} for the definitions of $\Sigma_{\sL}$ and $\Sigma_{\mathsf{Br}}$):
\begin{enumerate}
 \setlength\itemsep{.3em}
    \item The tight-span $\TS(\sw)$ is vertex-intersecting.  
    \item The dual graph $\Gamma(\sw)$ is a tree. 
    \item The subcomplex $\Sigma_{\sL}$ of $\TS(\sw)$  is vertex-intersecting.  
    \item The subcomplex $\Sigma_{\mathsf{Br}}$ of $\TS(\sw)$ is vertex-intersecting.     
    \item The subcomplex   of $\Sigma_{\sL}$ obtained by removing all fins of $\Sigma_{\sL}$ whose connecting path (see \S \ref{sec:limitsGraphsTS}) has length 1 is vertex-intersecting.  
    \item The subcomplex of  $\Sigma_{\sL}$ obtained by removing all fins of $\Sigma_{\sL}$ (of any connecting path length) is a tree.
\end{enumerate}

\noindent While these conditions are not mutually exclusive, we can still use this to separate all the subdivisions into 6 sets: set $\sG_j$ consists of those subdivisions that satisfy condition ($j$) but not ($i$) for $i< j$. 

\subsection{Case $\sG_1$} There are $13\,641$ subdivisions belonging to the set $\sG_1$.

\begin{proposition}
\label{prop:G1}
If $\sw \in \sG_1$, then $\Gr(\sw)$ is smooth and irreducible of dimension 15. 
\end{proposition}

\noindent This is a direct consequence of the following proposition.

\begin{proposition}
\label{prop:cvp}
Suppose $\sw \in \TGr_0(r,n)$ and $\TS(\sw)$ is basis-intersecting. Then $\Gr(\sw)$ is isomorphic to an open subvariety of $\A^{r(n-r)}$. In particular, $\init_{\sw}\Gr_0(r,n)$ and  $\Gr(\sw)$ are smooth and irreducible of dimension $r(n-r)$.
\end{proposition}

\begin{proof}
The tight-span $\TS(\sw)$ is basis-connecting by \cite[Proposition~C.11]{CoreyGrassmannians} (Appendix by Cueto).
By Proposition \ref{prop:coordRingSigma}, we have that $\Gr(\sw)$ is a locally-closed subscheme of $\A^{r(n-r)}$, in particular $\dim \Gr(\sw) \leq r(n-r)$. By Theorem \ref{thm:closedImmersion}, we have that $\dim \Gr(\sw) \geq r(n-r)$ as the dimension of $\init_{\sw}\Gr_0(r,n)$ is $r(n-r)$.  Therefore, $\dim \Gr(\sw) = r(n-r)$, and so it is isomorphic to an open subvariety of $\A^{r(n-r)}$. The last statement in the proposition follows from Corollary \ref{cor:limit2Init}. 
\end{proof}

\subsection{Case $\sG_2$} There are $215$ subdivisions belonging to the set $\sG_2$.

\begin{proposition}
\label{prop:allButOneGr}
Suppose $\Gamma \subset \Gamma(\sw)$ is a tree, and $\sQ_v$ is $B$--maximal for all $v\in V(\Gamma)$  except (possibly) one, say $\sQ_{v_0}$. If $\Gr(\sQ_{v_0})$ is smooth and irreducible, then so is $\varprojlim_{\Gamma} \Gr$.
\end{proposition}

\begin{proof}
The diagram $\Gr:\Gamma \to \Csch$ satisfies the hypotheses of Proposition \ref{prop:limitTreeAllButOne} by Proposition \ref{prop:conditionIdealTrivial}, and so the limit  $\varprojlim_{\Gamma} \Gr$ is smooth and irreducible. 
\end{proof}

\begin{proposition}
\label{prop:G2}
If $\sw$ represents a combinatorial type in $\sG_2$ other than $\spCone$, then $\Gr(\sw)$ is smooth and irreducible of dimension 15. 
\end{proposition}

\begin{proof}
    There are 210 combinatorial types $\sw$ so that $\Gamma=\Gamma(\sw)$ satisfies the hypothesis of Proposition \ref{prop:allButOneGr}, and so $\Gr(\sw)$ is smooth and irreducible in these cases. For the 4 remaining cases, all $\sQ_{v}$ have 2 or more parallel elements for all but at most one $v\in V(\Gamma)$. These diagrams satisfy the conditions of Proposition \ref{prop:limitTreeAllButOne} by \cite[Proposition~5.4]{CoreyGrassmannians} and [\textit{loc. cit.} Lemma C.5] (Appendix by Cueto). Finally, we compute the dimensions of each limit $\Gr(\sw)$ using Formula \ref{eq:dimTree}, and verify that each dimension is $15$. 
\end{proof}

\begin{example}
Here we use an explicit example as a demonstration of Proposition \ref{prop:allButOneGr}.
Let $\sw \in \TGr_{0}(3,8)$ be
\begin{equation*}
    \sw = \se_{126} + \se_{234} + \se_{237} + 2\, \se_{238} + \se_{247} + \se_{248} + \se_{278} + \se_{347} + \se_{348} + \se_{378} + 2\, \se_{478} + \se_{568}
\end{equation*}
The tight-span $\TS(\sw)$ of the subdivision $\pQ(\sw)$  is the tree given in Figure \ref{fig:tree_and_matroids}.

\begin{figure}[h]
    \centering
  \includegraphics[height = 2.5cm]{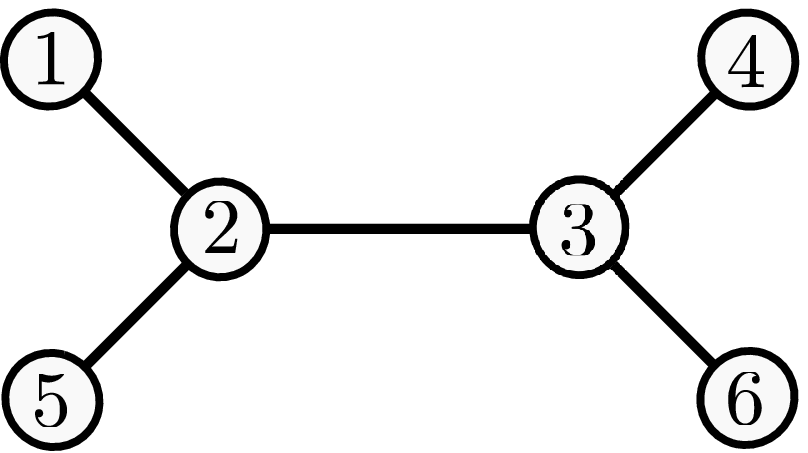}\space\space\space\space\space\space\space\includegraphics[height = 2.75cm]{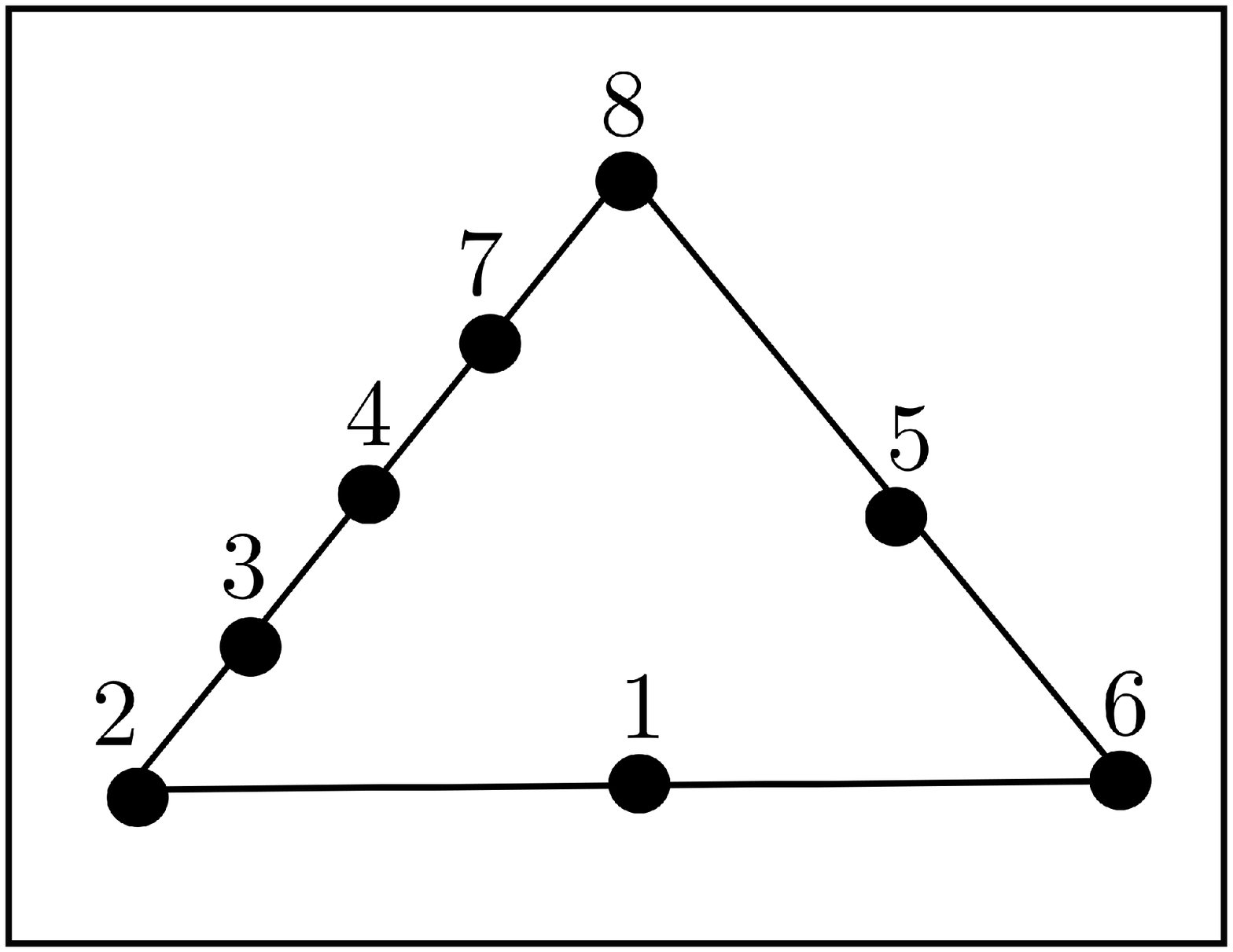}\space\space\includegraphics[height = 2.75cm]{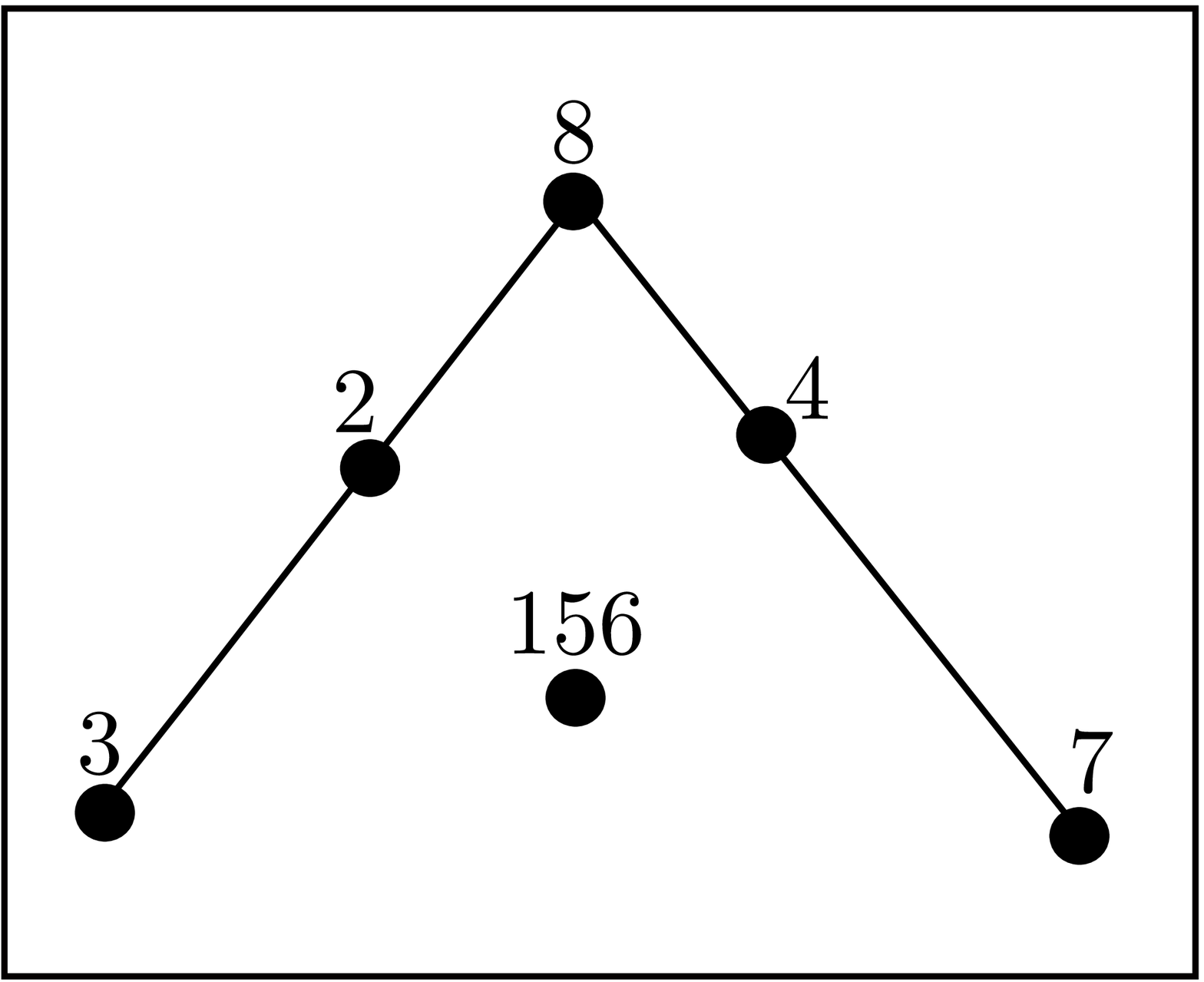}\space\space\includegraphics[height = 2.75cm]{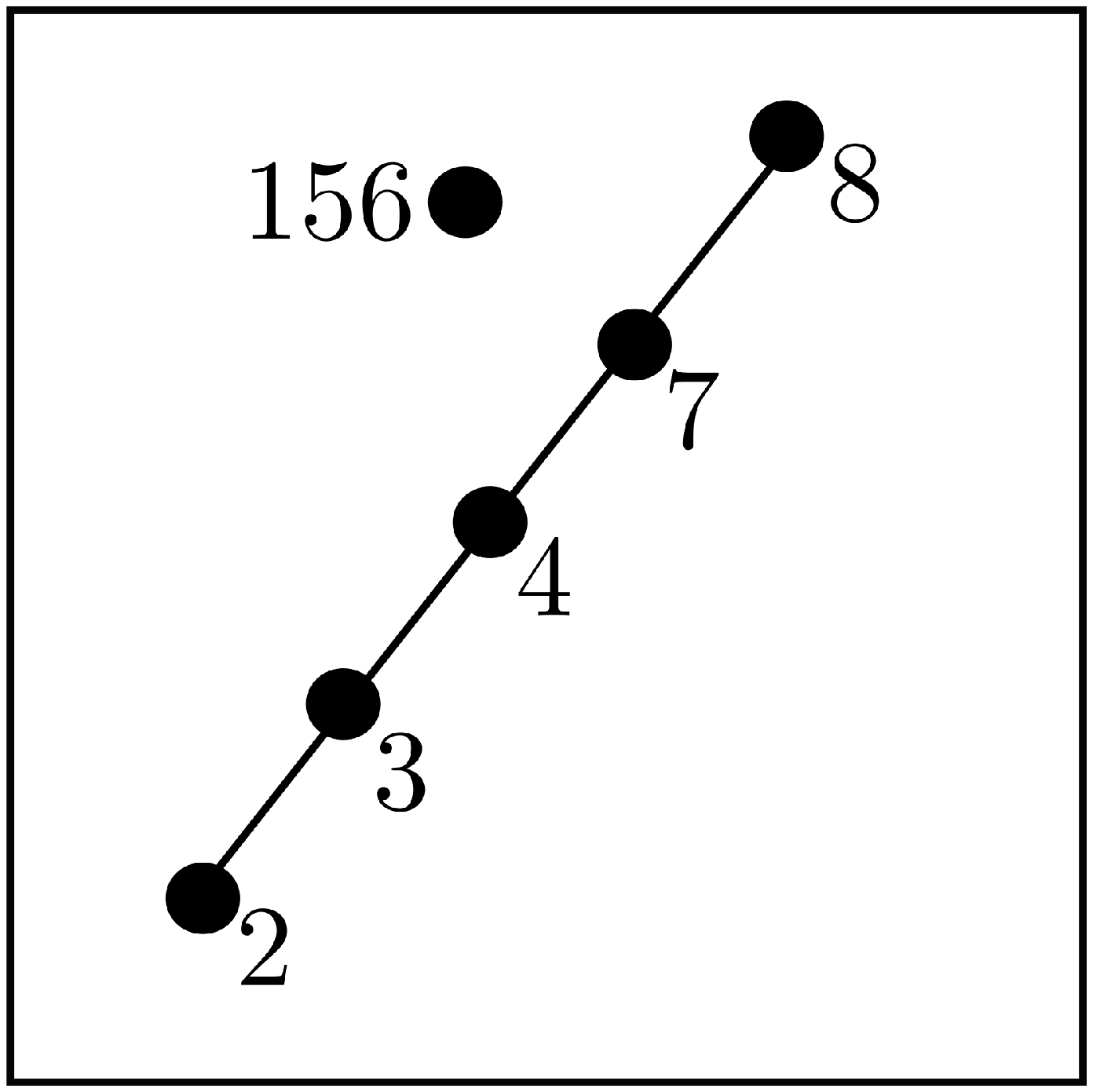}
    \caption{Dual graph of $\pQ(\sw)$ and the matroids corresponding to $\sQ_2,\sQ_3$ and $\sQ_{2,3}$}
    \label{fig:tree_and_matroids}
\end{figure}
\noindent The matroids $\sQ_2$ and $\sQ_3$ are illustrated in the middle of Figure \ref{fig:tree_and_matroids}. The remaining $\sQ_i$ are
\begin{equation*}
    \sQ_1 = \sU(34578,1,2,6), \; \sQ_4 = \sU(14567,2,3,8), \; \sQ_5 = \sU(12347,5,6,8),\;  \sQ_6 = \sU(12356,4,7,8).
\end{equation*}
Denote by $\sQ_{i,j}$ the matroid of the edge between the vertices corresponding to $\sQ_{i}$ and $\sQ_j$. The matroid $\sQ_{2,3}$ is illustrated by the right picture in Figure \ref{fig:tree_and_matroids}, and the remaining matroids  $\sQ_{i,j}$ for $(i,j) \neq (2,3)$ are isomorphic to $\sU'(a,b,c;[8]\setminus abc)$.  The dimensions of the thin Schubert cells $\Gr(\sQ_i)$ and $\Gr(\sQ_{i,j})$ are
{\footnotesize\begin{equation*}
    \dim \Gr(\sQ_i) = 
\begin{cases}
7 &\text{ if }i\in\{1,4,5,6\}\\
10 &\text{ if } i = 2\\
9 &\text{ if } i = 3\\
\end{cases} 
\hspace{20pt}
\dim \Gr(\sQ_{i,j}) = 
\begin{cases}
6 &\text{ if }(i,j)\in\{(1,2),(2,5),(3,4),(3,6)\}\\
8 &\text{ if } (i,j) = (2,3)\\
\end{cases}
\end{equation*}}
With $\mu_1 = \{2,6,8\}$ and $\mu_2 = \{1,4,7\}$, the matroids  $\sQ_{1}, \sQ_{2}, \sQ_{4}, \sQ_{5}$ are $(B,\mu_1)$-maximal and $\sQ_6$ is $(B,\mu_2)$-maximal, but $\sQ_3$ is not $B$-maximal. As $\sQ_3$ is not isomorphic to $\spMat$, we see that $\Gr(\sQ_3)$ is smooth and irreducible. Therefore, $\Gr(\sw)$ is smooth and irreducible of dimension 15 by Proposition \ref{prop:limitTreeAllButOne}, and so $\init_{\sw}\Gr_0(3,8)$ is smooth and irreducible by Corollary \ref{cor:limit2Init}. 
\end{example}

\subsection{Case $\sG_3$} There are $28\,227$ subdivisions belonging to the set $\sG_3$. Given such a combinatorial type $\sw$, let $\Sigma = \TS(\sw)$,  and $\Sigma_{\sL} \subset \Sigma$ the subcomplex obtained by removing all leaf-vertices of $\Sigma$ together with their adjacent edges. Set
\begin{equation}
\label{eq:leavesVE}
    Y_{\Leaf} = \prod_{\substack{v \text{ leaf} \\ \text{vertex}}} \Gr(\sQ_v) \hspace{20pt} Y_{\Leaf}' = \prod_{\substack{e \text{ leaf} \\ \text{edge}}} \Gr(\sQ_e) 
\end{equation}
Then
\begin{equation}
\label{eq:deleaf}
    \Gr(\sw) \; \cong \; \varprojlim_{\Sigma} \Gr \; \cong \; \varprojlim_{\Sigma_{\sL}} \Gr \times_{Y_{\Leaf}'} Y_{\Leaf}.
\end{equation}

\begin{proposition}
\label{prop:limitcvpLeaf}
If  $\varprojlim_{\Sigma_{\sL}} \Gr $ is smooth and irreducible, then $\Gr(\sw)$ is smooth and irreducible of dimension
\begin{equation}
\label{eq:dimLimitCVPLeaf}
    \dim \Gr(\sw) = \dim \varprojlim_{\Sigma_{\sL}} \Gr \; - \; \sum_{\substack{e \text{ leaf} \\ \text{edge}}} \dim \Gr(\sQ_e) \; + \; \sum_{\substack{v \text{ leaf} \\ \text{vertex}}} \dim \Gr(\sQ_v)
\end{equation}
\end{proposition}

\begin{proof}
For each leaf-pair $(v,e)$, the morphism $\Gr(\sQ_v) \to \Gr(\sQ_e)$ is smooth and dominant with connected fibers by Proposition \ref{prop:oneInternalFacet}, and therefore  $Y_{\Leaf}\to Y_{\Leaf}'$ also an SDC-morphism. The proposition now follows from Proposition \ref{prop:limitTreeAllButOne}. 
\end{proof}

\begin{proposition}
\label{prop:deLeafBasisConnecting}
For any combinatorial type $\sw$, the subcomplex $\Sigma_{\sL}$ is vertex-connecting.
\end{proposition}
\begin{proof}
The complex $\TS(\sw)$ is vertex-connecting by \cite[Lemma~C.11]{CoreyGrassmannians} (Appendix by Cueto). Now apply Proposition \ref{prop:vertexConnectingLeaf}. 
\end{proof}

By Propositions \ref{prop:coordRingSigma} and \ref{prop:deLeafBasisConnecting}, the coordinate ring of $\varprojlim_{\Sigma_{\sL}} R_{\sQ}$ is isomorphic to $R_{\Sigma_{\sL}}$. Thus, we need some techniques for determining when a ring of the form $R_{\Sigma}$ is a regular (i.e., its local rings are regular) integral domain. We work in the following more general context. Let $M = \Z^b$  with standard basis $\epsilon_{1}^{*},\ldots,\epsilon_{b}^{*}$, let $x_i \in \C[M]$ be the variable corresponding to $\epsilon_i^*$. More generally, let $x^{\su} \in \C[M]$ be the monomial corresponding to $\su\in M$. Let $B = \C[x_{1},\ldots,x_{b}]$. Given $f  \in B$, let $\su(f)$ be 
\begin{equation*}
    \su(f) = \sum_{c_{\su} \neq 0} \su \hspace{20pt} \text{where} \hspace{20pt} f = \sum_{\su} c_{\su}x^{\su}.
\end{equation*}
Let $S\subset B$ a finitely generated multiplicative semigroup whose saturation contains the variables $x_1,\ldots,x_b$, and $I \subset B$ an ideal generated by distinct nonzero elements $f_{1},\ldots,f_a$; assume that  $a\leq b$ ($a=0$ if $I=\langle 0 \rangle$). Set $R = S^{-1}B/I$. Let $A(f_1,\ldots,f_a)$ be the $a\times b$ matrix with rows $\su(f_1),\ldots,\su(f_a)$. 

\begin{lemma}
\label{lem:upperTriangularEliminate}
If, after a permutation of the rows, the matrix $A(f_1,\ldots,f_a)$ has an upper-triangular $(a\times a)$-submatrix whose diagonal entries equal 1, then $S^{-1}B/I$ is a regular integral domain of Krull dimension $b-a$. 
\end{lemma}

\begin{proof}
We proceed by induction on $a$. If $a=0$, then $I=\langle 0 \rangle$, and the conclusion is clear. Now suppose that the lemma is true for values strictly less than $a$. After permuting the rows and columns (equivalently, permuting the $f_i$'s and $x_{j}$'s), we may assume that the first $a$ columns form an upper triangular matrix whose diagonal entries equal 1. This means that 
\begin{equation*}
    f_1 = g \cdot x_1 - h \hspace{15pt} \text{where} \hspace{15pt} g = c \prod_{i>1} x_i^{u_i}, \hspace{10pt} \text{and} \hspace{10pt} h \in \C[x_{2},\ldots,x_b]
\end{equation*}
Here, $c$ is a nonzero constant and the $u_i$ are nonnegative. Therefore, the assignment $x_1\mapsto h/g$, and $x_i \mapsto x_i$  for $i>1$ defines a ring isomorphism
\begin{equation*}
    R \to \bar{S}^{-1} \C[x_2,\ldots,x_b] / \langle f_2,\ldots,f_b \rangle 
\end{equation*}
where $\bar{S}$ is the image of $S$ under the above substitution. The lemma now follows from the inductive hypothesis. 
\end{proof}

Lemma \ref{lem:upperTriangularEliminate} suffices to determine that $R_{\Sigma_{\sL}}$ is a regular integral domain for most cases in $\sG_3$. For the remaining, we use the following Algorithm. Given $f\in B$, denote by $\deg_{k}(f)$ the degree of $x_k$ in $f$, i.e., the highest power of $x_k$ appearing in $f$. If $\deg_{k}(f) = 1$, then denote by $\coef(f;x_k)$ the coefficient of $x_k$ in $f$.
We say that $f$ is \textit{reduced} if $f$ has no unit other than elements of $\C^*$ as a factor.  Because $B$ is a unique factorization domain, if $f\in S^{-1}B$ is not a unit, then there is a unique expression (up to a factor lying in $\C^*$) $f = g\cdot  \bar{f}$ such that $g\in S$ and $\bar{f}$ lie in $B$ and has no unit (other than an element of $\C^*$) as a factor.

\begin{algm}\label{algm:reduceIdeal}
\leavevmode

\noindent \textit{Input}: The polynomial ring $B = \C[x_1,\ldots,x_b]$, $S\subset B$ a multiplicative semigroup containing all monomials,  and $\pF \subset B$ a finite set of reduced generators of $I\subset S^{-1}B$. Set $R = S^{-1} B / I$.  

\medskip

\noindent \textit{Output}: A new presentation of $R$, say $R \cong \bar{S}^{-1}\bar{B} / \bar{I}$, and a finite set $\bar{\pF}\subset \bar{B}$ of reduced generators of $\bar{I}\subset \bar{S}^{-1} \bar{B}$ such that 
\begin{itemize}
    \item[-] $\bar{B}$ is a polynomial ring on a subset of the $x_1,\ldots,x_b$
    \item[-] for each $x_k\in \bar{B}$ and $f \in \bar{\pF}$, we have $\deg_{k} f \neq 1$ or $\deg_{k} f = 1$ but $\coef(f;x_k) \notin S$.
\end{itemize}
In particular, if $\bar{I} = \langle 0 \rangle$, then $R$ is a regular integral domain. 

\medskip

\noindent Find $f\in \pF$ and $x_k\in B$ such that $\deg_{k} f = 1$ and $\coef(f;x_k) \in S$. This means that
\begin{equation*}
f = g \cdot x_k - h
\end{equation*}
for some $g\in S$ and $h \in B$ such that $\deg_k g = 0$ and $\deg_k h = 0$. Then perform the substitution $x_k \mapsto h/g$ followed by a reduction for all elements of $\pF$, and just the substitution for the elements of $S$; this  produces new sets $\bar{\pF}$ and $\bar{S}$. Set $\bar{B} = \C[x_j \, : \, j\neq k]$.  Now repeat this procedure with $\pF\gets \bar{\pF}$, $S \gets \bar{S}$ and $B \gets \bar{B}$. 

\medskip

\noindent Once no such pair $(f,x_k)$ exists, this function returns $\pF$ and the presentation $R = S^{-1} B / I$. 
\end{algm}

\begin{proposition}
\label{prop:G3}
If $\sw \in \sG_5$, then $\Gr(\sw)$ is smooth and irreducible of dimension 15. \end{proposition}

\begin{proof}
For $26\, 187$ combinatorial types $\sw\in \sG_3$, there is a presentation of the ideal $I_{\Sigma_{\sL}}$ that satisfies the conditions of Lemma \ref{lem:upperTriangularEliminate}. So $\varprojlim_{\Sigma_{\sL}}\Gr$ is smooth and irreducible. For the remaining $2040$ combinatorial types, we apply Algorithm \ref{algm:reduceIdeal}. In each case, this function returns a presentation of $R_{\Sigma_{\sL}}$ whose ideal is $\langle 0 \rangle$, whence $R_{\Sigma_{\sL}}$, and therefore $R_{\Sigma}$, is a regular integral domain. In either case, we get that $\Gr(\sw)$ is smooth and irreducible by Proposition \ref{prop:limitcvpLeaf}. 

Finally, we compute the dimensions of each $\Gr(\sw)$ using Formula \eqref{eq:dimLimitCVPLeaf}, verify that each dimension is $15$. 
\end{proof}

\subsection{Case $\sG_4$}\label{sec:G4} 
There are $483$ combinatorial types belonging to the set $\sG_4$. 

\begin{proposition}
\label{prop:G4}
If $\sw \in \sG_4$, then $\Gr(\sw)$ is smooth and irreducible of dimension 15. 
\end{proposition}

\begin{proof}
For any $\sw\in \sG_4$, it happens to be the case that any branch of $\TS(\sw)$ is a path, and so each of its vertices has 2 or fewer neighbors. For such a vertex $v$, its dual polytope $\Delta_v$ has 2 or fewer internal facets. Thus, each such $\Gr(\sQ_v)$ is smooth and irreducible and the maps from $\Gr(\sQ_v)$ corresponding to the edges adjacent to $v$ are SDC-morphisms by Proposition \ref{prop:oneInternalFacet}. By Proposition \ref{prop:removeBranchLim}, to show that $\Gr(\sw)$ is smooth and irreducible, it suffices to show that $\varprojlim_{\Sigma_{\mathsf{Br}}} \Gr$ is smooth and irreducible. For $15$ combinatorial types, there is a presentation of the ideal $I_{\Sigma_{\mathsf{Br}}}$ that satisfies the conditions of Lemma \ref{lem:upperTriangularEliminate}. For the remaining $468$ combinatorial types, we apply Algorithm  \ref{algm:reduceIdeal}, which always returns the ideal $\langle 0 \rangle$.  In either case, we see that $R_{\Sigma_{\mathsf{Br}}}$ is a regular integral domain, as required. Finally, we verify that the dimensions of the $\Gr(\sw)$ are all 15 using the formula
\begin{equation*}
    \dim \Gr(\sw) = \dim \varprojlim_{\Sigma_{\mathsf{Br}}} \Gr - \sum_{\substack{e \text{ branch } \\ \text{edge} }} \dim \Gr(\sQ_e) + \sum_{\substack{v \text{ branch } \\ \text{vertex} }} \dim \Gr(\sQ_v) 
\end{equation*}
which follows from Formula \ref{eq:dimDeBranch}. 
\end{proof}

\subsection{Case $\sG_5$} There are $14\,389$  combinatorial types belonging to the set $\sG_5$. Similar to the case $\sG_3$, given a combinatorial type $\sw$, let $\Sigma = \TS(\sw)$ and $\Sigma_{\sL} \subset \Sigma$ the subcomplex obtained by removing all leaf-vertices of $\Sigma$ together with their adjacent edges. Let $Y_{\Leaf}$ and $Y_{\Leaf}'$ be the schemes in Formula \ref{eq:leavesVE}. Formula \ref{eq:deleaf} holds, but now $\varprojlim_{\Sigma_{\sL}} \Gr$ cannot be evaluated using a presentation provided by Proposition \ref{prop:coordRingSigma} since $\Sigma_{\sL}$ is not vertex-intersecting. 

Recall that a fin $\sF$ is $B$-maximal if there is a $\mu \in \sQ_{\sF}$ such that $\sQ$ is $B$-maximal for all $\sQ \in \pE_{\vertices}(\sF)$.  Given a collection of fins $\mathfrak{F}$ of $\Sigma_{\sL}$, let  $\Sigma_{\sL}(\mathfrak{F}) \subset \Sigma_{\sL}$ be the subcomplex obtained by removing the fins $\sF \in \mathfrak{F}$.  

\begin{proposition}\label{prop:fins}
If  $\varprojlim_{\Sigma_{\sL}(\mathfrak{F})}\Gr$ is smooth and irreducible and each fin $\sF \in \mathfrak{F}$ is $B$-maximal, then $\Gr(\sw)$ is smooth and irreducible of dimension
\begin{align}
    \nonumber \dim & \Gr(\sw) = \dim \varprojlim_{\Sigma_{\sL}(\mathfrak{F})} \Gr \; - \; \sum_{\substack{e \text{ leaf} \\ \text{edge}}} \dim \Gr(\sQ_e) \; + \sum_{\substack{v \text{ leaf} \\ \text{vertex}}} \dim \Gr(\sQ_v) \; \\
    & +\sum_{\sF \in \mathfrak{F}} \;\left(\dim \Gr(\sQ_{\sF}) - \sum_{e \in \pE_{\edge}(\sF)} \dim \Gr(\sQ_e)  + \sum_{v \in \pE_{\vertices}(\sF)} \dim \Gr(\sQ_v)   \right) \label{eq:dimDefin}
\end{align}
\end{proposition}

\begin{proof}
By Propositions \ref{prop:addfin} and \ref{prop:finSDC}, we have that $\varprojlim_{\Sigma_{\sL}} \Gr$ is smooth and irreducible of dimension 
{\footnotesize
\begin{equation*}
    \dim \varprojlim_{\Sigma_{\sL}} \Gr = \dim \varprojlim_{\Sigma_{\sL}(\mathfrak{F})} \Gr \; +\sum_{\sF \in \mathfrak{F}} \;\left(\dim \Gr(\sQ_{\sF}) - \sum_{e \in \pE_{\edge}(\sF)} \dim \Gr(\sQ_e)  + \sum_{v \in \pE_{\vertices}(\sF)} \dim \Gr(\sQ_v)   \right)
\end{equation*}
}
Now apply Propositions \ref{prop:limitcvpLeaf}. 
\end{proof}

\begin{proposition}
\label{prop:G5}
If $\sw \in \sG_5$, then $\Gr(\sw)$ is smooth and irreducible of dimension 15. 
\end{proposition}

\begin{proof}
Fix $\sw \in \sG_5$, $\Sigma = \TS(\sw)$, and $\Sigma_{\sL} \subset \Sigma$ the subcomplex obtained by removing all leaf vertices and edges. Let $\mathfrak{\sF}$ be the set of fins of $\Sigma_{\sL}$ that have connecting path length 1. We verify that each $\sF \in \mathfrak{F}$ is $B$-maximal. In this case, the set $\Sigma_{\sL}(\mathfrak{F})$ is vertex-connecting by Proposition \ref{prop:vertexConnectingFin} and  vertex-intersecting by hypothesis, and therefore we may compute the coordinate ring of $\varprojlim_{\Sigma_{\sL}(\mathfrak{F})} \Gr$ using Proposition \ref{prop:coordRingSigma}. We verify that these rings are smooth integral domains as in the proof of Proposition \ref{prop:G3}. Finally, we verify that $\dim \Gr(\sw) = 15$ using Formula \eqref{eq:dimDefin}. The proposition now follows from Corollary \ref{cor:limit2Init}. 
\end{proof}

\begin{example}
Let 
\begin{align*}
    \sw =& \se_{124} + \se_{125} + 2\se_{126} + 3\se_{137} + 2\se_{138} + 2\se_{145} + \se_{146} + \se_{156}+2\se_{178}+\se_{234}+\se_{245}+\se_{246}+\\
    &2\se_{247} + 2\se_{248} + \se_{256} + 
3\se_{278} + \se_{356} + 2\se_{378} + \se_{456} + 3\se_{478} + 2\se_{567}+2\se_{568}+2\se_{578} + 2\se_{678}.
\end{align*}
The tight span $\TS(\sw)$ of $\pQ(\sw)$ is given in \ref{fig:mantis}. We also see the matroid corresponding to the node $v_2$ of $\TS(\sw)$ that did not appear in Figure $\ref{fig:matroidsUVW}$.
\begin{figure}
    \centering
    \includegraphics[height = 5cm]{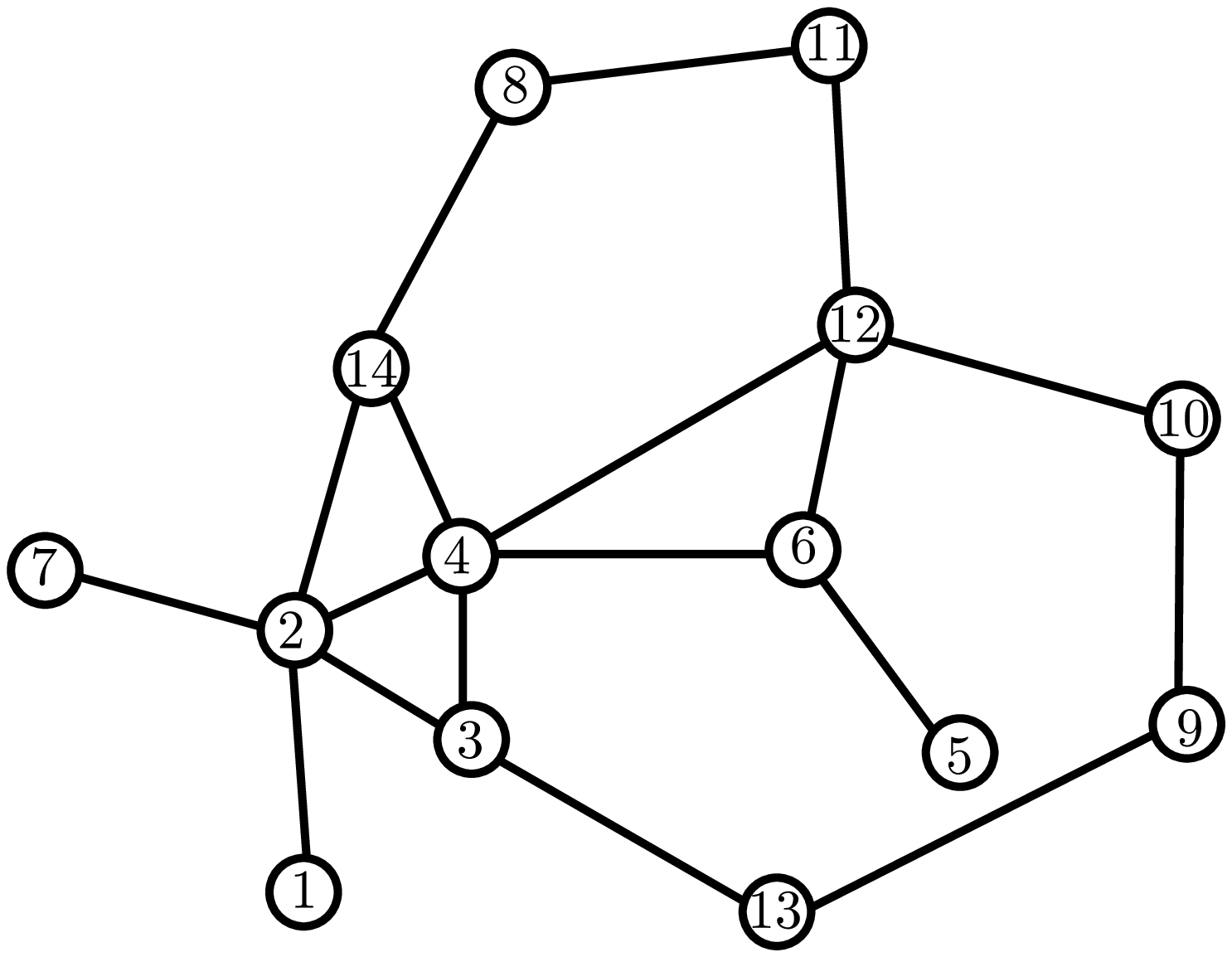}\space\space\space\space\space \includegraphics[height = 4cm]{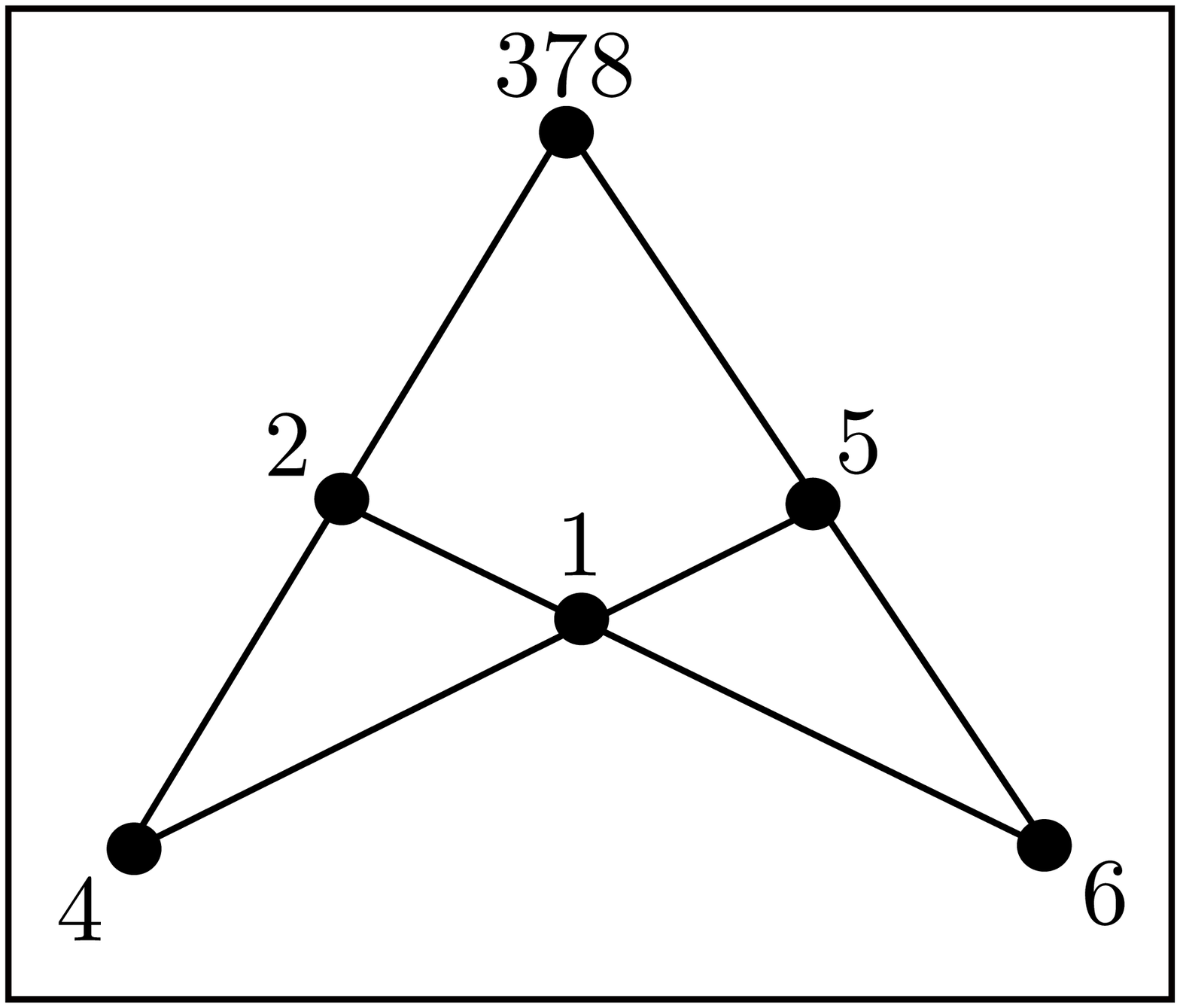}
    \caption{Dual graph of $\pQ(\sw)$ and the matroid corresponding to $\sQ_2$}
    \label{fig:mantis}
\end{figure}

The matroids $\sQ_i$ corresponding to the vertices $v_i$ of $\Gamma(\sw)$ are 
\begin{equation*}
\begin{array}{lll}
     \sQ_1=\sU(34578,1,2,6) &\sQ_3= \sW(78;3,156;2,4)  &\sQ_4 = \sW(1;24,56;3,78)\\
     \sQ_5=\sU(24568,1,3,7) &\sQ_6= \sW(7;8,2456;3,1) &\sQ_7 = \sU(23678,1,4,5)\\ \sQ_8 = \sU(1234,78,5,6) &\sQ_9= \sW(1356;7,8;2,4) & \sQ_{10} = \sU(1356,24,7,8)\\ 
     \sQ_{11} = \sW(1234;7,8;5,6)&\sQ_{12}= \sW(24;7,8;13,56) &\sQ_{13} = \sU(1356,78,2,4)\\
     \sQ_{14} = \sW(78;3,124;5,6) & &\\
\end{array}    
\end{equation*}
When $v_{i}$ and $v_j$ share an edge, the matroid of the edge is denoted by $\sQ_{i,j}$; in each case, $\sQ_{i,j}$ is isomorphic to a matroid of the form $\sU'$. Then Proposition \ref{prop:TSCUVW}, and direct computation for the $i=2$ case, tells us the dimensions of the thin Schubert cells corresponding to the vertices and edges of $\TS(\sw)$ are given by

\begin{equation*}
    \dim \Gr(\sQ_i) = 7 \text{ for } 1\leq i \leq 14 \hspace{15pt} \dim \Gr(\sQ_{i,j}) = 6\text{ for }(v_i,v_j)\in E(\Gamma(\sw))
\end{equation*}

\noindent Moreover, one can check that $\dim\Gr(\sQ_{\sF}) = 5$ for the unique fin $\sF=\langle v_4,v_6,v_{12}\rangle$ of $\TS(\sw)$. Finally, we compute $R_{\Sigma_{\sL}(\mathfrak{F})}$, to obtain the dimension of $\varprojlim_{\Sigma(\mathfrak{F})}$. We compute
\begin{equation*}
    B_{\Sigma_{\sL}(\mathfrak{F})} = \C[x_{11},x_{12},x_{13},x_{14},x_{15},x_{21},x_{22},x_{23},x_{24},x_{25},x_{32},x_{33},x_{34},x_{35}]
\end{equation*}
the ideal is
\begin{equation*}
    I_{\Sigma_{\sL}(\mathfrak{F})} = \langle x_{12}x_{23}-x_{11}x_{22},\; x_{12}x_{33}x_{24}+x_{22}x_{13}x_{34},\; x_{11}x_{33}x_{24}+x_{21}x_{13}x_{34}\rangle,
\end{equation*}
and $S_{\Sigma_{\sL}(\mathfrak{F})}$ is the semigroup of all monomials of $B_{\Sigma_{\sL}(\mathfrak{F})}$. One can check that the condition of Lemma \ref{lem:upperTriangularEliminate} does not hold. Thus we apply Algorithm \ref{algm:reduceIdeal}. That is, we obtain \begin{equation*}
    x_{12} \equiv \frac{x_{11}x_{22}}{x_{21}} \hspace{10pt} \text{ and } \hspace{10pt} x_{34} \equiv -\frac{x_{11}x_{33}x_{24}}{x_{21}x_{13}} \mod I_{\Sigma_{\sL}(\mathfrak{F})}. 
\end{equation*}

Observe that with these substitutions, we have $f_3 = 0$, and that the coordinate ring of
$\varprojlim_{\Sigma(\mathfrak{F})} \Gr$ reduces to $\C[x_{11}^{\pm},x_{13}^{\pm},x_{14}^{\pm},x_{15}^{\pm},x_{21}^{\pm},x_{22}^{\pm},x_{23}^{\pm},x_{24}^{\pm},x_{25}^{\pm},x_{32}^{\pm},x_{33}^{\pm},x_{35}^{\pm}]$. So we have that $\varprojlim_{\Sigma(\mathfrak{F})} \Gr$ is smooth, irreducible and of dimension 12.

We also have the fin $\sF$ is $B$-maximal, in the sense that $\sQ_6$ is $(B,\mu)$-maximal where $\mu=\{1,2,7\} \in \sQ_{\sF}$. By Proposition \ref{prop:fins} $\Gr(\sw)$ is smooth and irreducible. One may use Formula \eqref{eq:dimDefin} to verify that $\dim \Gr(\sw) = 15$. Hence, by Corollary \ref{cor:limit2Init}, $\init_{\sw}\Gr_{0}(3,8)$ is smooth and irreducible.
\end{example}

\subsection{Case $\sG_6$} 
There are $389$ combinatorial types in $\sG_6$.  For each combinatorial type $\sw \in \sG_6$, the dual graph $\Gamma(\sw)$ has a star-tree subgraph $\sT$ with 4 leaves; the matroid $\sC$ of the center node is isomorphic to $\sU(12,34,56,78)$, and the matroids $\sL_1,\sL_2,\sL_3,\sL_4$ of the 4 leaves are isomorphic to either $\sV(12,34,56;7,8)$ or $\sW(12; 34, 56; 7, 8)$. The set $\sG_6$ splits further into 5 groups $\sH_0, \sH_1, \sH_2, \sH_3, \sH_4$ where $\sw \in \sH_k$ if
\begin{equation*}
    |\{ m\, : \,  \sL_{m} \cong \sV(12,34,56;7,8) \}| = k.
\end{equation*}
The sizes of the $\sH_k$'s are
\begin{equation*}
    \begin{array}{lllll}
        |\sH_0| = 233 &  |\sH_1| = 127 & |\sH_2| = 25 & |\sH_3| = 3 & |\sH_4| = 1 
    \end{array}
\end{equation*}

\begin{proposition}
\label{prop:dimLimitT}
If $\sw \in \sH_k$, then the limit $\varprojlim_{\sT} \Gr$ is smooth and irreducible of dimension $11+k$. 
\end{proposition}

\begin{proof}
The matroid  $\sC_m$ corresponding to the edge between $\sC$ and $\sL_m$ is isomorphic to $\sU'(12,34,56;78)$, see Figure \ref{fig:matroidsUVW}. By Proposition \ref{prop:TSCUVW}, the thin Schubert cells $\Gr(\sC)$, $\Gr(\sC_m)$, $\Gr(\sL_m)$ are smooth, irreducible, and their dimensions are  $\dim \Gr(\sC)  = 7 $,  $\dim \Gr(\sC_m) = 6$ and
\begin{align*}
    \dim \Gr(\sL_m) = 8 \hspace{5pt} \text{ if } \sL_m  \cong \sV(12,34,56;7,8); \hspace{5pt} \dim \Gr(\sL_m) = 7 \hspace{5pt} \text{ if }  \sL_m \cong  \sW(12; 34, 56; 7, 8)
\end{align*}
Furthermore, the morphisms $\Gr(\sL_m) \to \Gr(\sC_m)$ are smooth and dominant with connected fibers by Proposition \ref{prop:TSCUVW}. The proposition now follows from Proposition \ref{prop:limitTreeAllButOne}. 
\end{proof}

\begin{proposition}
\label{prop:G6}
If $\sw \in \sG_6$, then $\Gr(\sw)$ is smooth and irreducible of dimension 15. 
\end{proposition}

\begin{proof}
Suppose $\sw \in \sH_k$, and let $\Sigma = \TS(\sw)$.  As usual, let $\Sigma_{\sL} \subset \Sigma$ be the subcomplex obtained by removing all leaf vertices and edges from $\Sigma$.  By a direct verification, $\Sigma_{\sL}$ consists of the tree $\sT$, and 6 fins, each of which has contact-length 2. In fact, there is a fin attached to $\sT$ at each pair of edges. We verify that each fin is $B$-maximal. Thus $\Gr(\sw)$ is smooth and irreducible by Propositions \ref{prop:limitFins} and \ref{prop:finSDC}. Finally, we verify that $\dim \Gr(\sw) = 15$ using Formula \eqref{eq:dimDefin}. The proposition now follows from Corollary \ref{cor:limit2Init}. 
\end{proof}

\begin{example}
\label{ex:G6}
The vector
{\smaller
\begin{align*}
    \sw = &2\,\se_{123} + 2\,\se_{124} + 2\,\se_{134} + 2\,\se_{135} + 2\,\se_{136} + 4\,\se_{137} + 3\,\se_{138} + \se_{156} + \se_{178} + 2\,\se_{234} + \se_{245}   \\
    &+\se_{246} + \se_{247}  + \se_{248} + \se_{256} + \se_{278}  + \se_{356} + \se_{378} + \se_{456} + \se_{478} +  2\se_{567}  + 2\se_{568} +  2\se_{578} + 2\se_{678} 
\end{align*}
}

\noindent represents a combinatorial type in $\sH_0$, the tight span $\TS(\sw)$ is shown in Figure \ref{fig:exampleG6}, together with the matroids corresponding to the nodes of $\TS(\sw)$. The tree $\sT$ consists of the vertices $v_{2},v_{7},v_{9},v_{10},v_{15}$, and $\varprojlim_{\sT}\Gr$ is smooth and irreducible of dimension 11 by Proposition \ref{prop:dimLimitT}.  The subcomplex $\Sigma_{\sL} \subset \TS(\sw)$ obtained by removing all leaves is obtained by removing the vertex $v_3$ and its adjacent edge. The fins of $\Sigma_{\sL}$ are:
\begin{equation*}
\begin{array}{lll}
    \sF_1 = \langle v_{2}, v_{4}, v_{9}, v_{15}\rangle & \sF_2 = \langle v_{2}, v_{7}, v_{8}, v_{9}\rangle & \sF_3 = \langle v_{6}, v_{7}, v_{9}, v_{10}\rangle  \\
    \sF_4 = \langle v_{9}, v_{10}, v_{11}, v_{15}\rangle & \sF_5 = \langle v_{1}, v_{2}, v_{5}, v_{9}, v_{10}\rangle & \sF_6 = \langle v_{7}, v_{9}, v_{12}, v_{13}, v_{14}, v_{15}\rangle.
\end{array}
\end{equation*}
With $\mu_1 = \{1,5,7\}$ and $\mu_2 = \{1,2,5\}$, the fins $\sF_1$, $\sF_4$, $\sF_6$ are $(B,\mu_1)$-maximal, and the fins $\sF_2$, $\sF_3$, $\sF_5$ are $(B,\mu_2)$-maximal. Therefore, $\Gr(\sw)$ is smooth and irreducible by Propositions \ref{prop:limitFins} and \ref{prop:finSDC}. Using this and Proposition \ref{prop:TSCUVW}, one may verify that $\dim \Gr(\sw) = 15$, and therefore $\init_{\sw} \Gr_0(3,8)$ is smooth and irreducible by Corollary \ref{cor:limit2Init}. 
\end{example}

\begin{figure}
\begin{minipage}{0.4\textwidth}
\includegraphics[width=\linewidth,keepaspectratio=true]{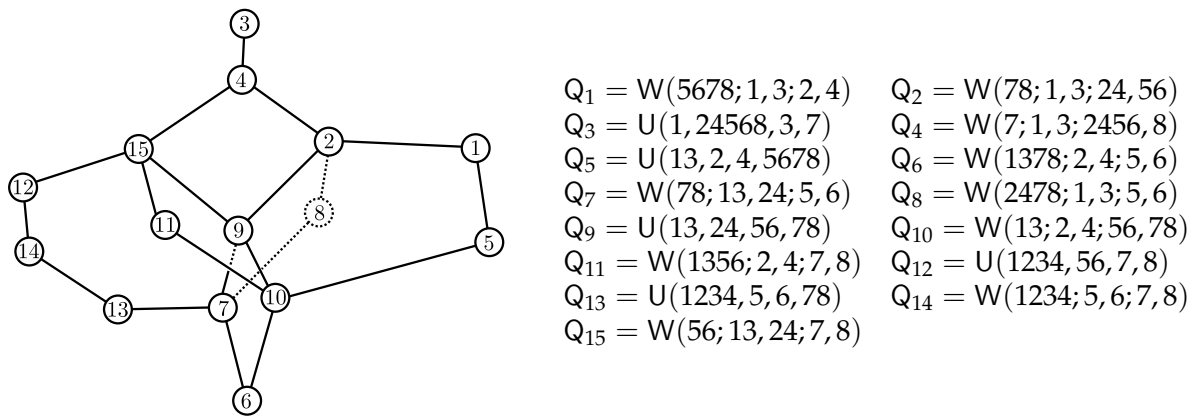}
\end{minipage}
\begin{minipage}{0.6\textwidth}
\small
\begin{align*}
    \begin{array}{ll}
    \sQ_1 = \sW(5678;1,3;2,4) & \sQ_2 = \sW(78;1,3;24,56) \\
    \sQ_3 = \sU(1,24568,3,7) & \sQ_4 = \sW(7;1,3;2456,8) \\
    \sQ_5 = \sU(13,2,4,5678) & \sQ_6 = \sW(1378;2,4;5,6) \\
    \sQ_7 = \sW(78;13,24;5,6) & \sQ_8 = \sW(2478;1,3;5,6) \\
    \sQ_9 = \sU(13,24,56,78) & \sQ_{10} = \sW(13;2,4;56,78) \\
    \sQ_{11} = \sW(1356;2,4;7,8) & \sQ_{12} = \sU(1234,56,7,8) \\
    \sQ_{13} = \sU(1234,5,6,78) & \sQ_{14} = \sW(1234;5,6;7,8) \\
    \sQ_{15} = \sW(56;13,24;7,8) &
    \end{array}
\end{align*}
\null
\par\xdef\tpd{\the\prevdepth}
\end{minipage}
\caption{Left: the tight span $\TS(\sw)$; right: the matroids corresponding to the nodes of $\TS(\sw)$}
\label{fig:exampleG6}
\end{figure}

\subsection{Completion of the proofs}

The following is a compilation of Propositions \ref{prop:specialLimit}, \ref{prop:G1}, \ref{prop:G2}, \ref{prop:G3}, \ref{prop:G4}, \ref{prop:G5},  \ref{prop:G6}, and Corollary \ref{cor:limit2Init}. 

\begin{theorem}
\label{thm:smoothLimits}
Suppose $\sw \in \TGr_0(3,8)$.
\begin{enumerate}
    \item The inverse limit $\Gr(\sw)$ is smooth of dimension 15.
    \item If $\sw$ is contained in the relative interior of any cone in the $\Sn{8}$-orbit of $\spCone$, then $\Gr(\sw)$ has 2 connected components. Otherwise, $\Gr(\sw)$ is irreducible. \item The closed immersion $\init_{\sw}\Gr_0(3,8)\hookrightarrow \Gr(\sw)$ is an isomorphism.
\end{enumerate}
\end{theorem}

\begin{proof}[Proofs of Theorems \ref{thm:schoenIntro} and \ref{thm:connectedInitDegIntro}]
These follow from Theorem \ref{thm:smoothLimits}. 
\end{proof}

\section{The Chow quotient $\chow{\Gr(3,8)}{H}$}
\label{sec:Chow}

Recall from the introduction that the diagonal torus of $\GL(n)$ acts on $\C^{n}$ by scaling the coordinates, and this induces an action of the diagonal torus $H\subset \PGL(n)$ on $\Gr(r,n)$. This restricts to a free action on $\Gr_0(r,n)$, and under the Gelfand-MacPherson correspondence, the quotient $X(r,n):= \Gr_0(r,n) / H$ coincides with the moduli space of projective-equivalence classes of $r$ hyperplanes in $\P^{n-1}$ in linear general position.

Throughout this section, let $\sQ = \binom{[n]}{r}$ be the uniform $(r,n)$--matroid. Let $L\subset N(\sQ)$ be the saturated subgroup from Formula \eqref{eq:lineality}, and $T_{L}\subset T(\sQ)$ its corresponding torus. Then $T_{L} \cong H$ and the action of  $T_{L}$ on $\P(\wedge^r\C^n)$ through $T(\sQ)$ restricts to the action of $H$ on $\Gr(r,n)$. Thus, we have an inclusion of Chow quotients
\begin{equation*}
    \chow{\Gr(r,n)}{H} \hookrightarrow \chow{\P(\wedge^r\C^n)}{H}. 
\end{equation*}
By \cite{KapranovSturmfelsZelevinsky}, the normalization of the Chow quotient $\chow{\P(\wedge^r\C^n)}{H}$ is the toric variety of the pointed fan $\pS(r,n)/L_{\R}$, where $\pS(r,n)$ is the secondary fan of $\Delta(r,n)$. 

Recall that $\pS_{\trop}(3,8)$ is the subfan of $\pS(3,8)$ whose support is $\TGr_0(3,8)$. As $\Trop X(3,8)$ is $\TGr_0(3,8) / L_{\R}$, the closure of $X(3,8)$ in the toric variety $X(\pS_{\trop}(3,8)/L_{\R})$ coincides with its closure in $\chow{\P(\wedge^r\C^n)}{H}$ by \cite[Proposition~2.3]{Tevelev}, and we denote this by $\overline{X}(3,8)$. This space has the same normalization as $\chow{\Gr(3,8)}{H}$. 
By Theorem \ref{thm:schoenIntro} and the fact that $\init_{\sw} \Gr_0(r,n) \cong \init_{\bar{\sw}} X(r,n) \times H$ (where $\bar{\sw}$ is the image of $\sw$ under the quotient $N(\sQ)_{\R} \to (N(\sQ)/L)_{\R}$), we conclude that $X(3,8)$ is sch\"on, and hence $\overline{X}(3,8)$ is a sch\"on compactification  of $X(3,8)$ \cite[Theorem~1.5]{LuxtonQu}. In particular, $\overline{X}(3,8)$ is normal and the boundary $B:=\overline{X}(3,8)\setminus X(3,8)$ is a divisor. 

\begin{theorem} 
\label{thm:KXBAmple}
The log-canonical divisor $K_{\overline{X}} + B$ of $\overline{X} = \overline{X}(3,8)$ is ample. In particular, $\overline{X}(3,8)$ is the log-canonical compactification of $X(3,8)$. 
\end{theorem}

\begin{proof}
By \cite[Theorem~9.1]{HackingKeelTevelev2009}, it suffices to show that each irreducible stratum is log minimal. As $\overline{X}(3,8)$ is a sch\"on compactification, its strata are also sch\"on by Formula \eqref{eq:initStratum} and \cite[Corollary~4.2.10]{MaclaganSturmfels2015}. By  \cite[Theorem~3.1]{HackingKeelTevelev2009}, a sch\"on very affine variety $X\subset T$ is log minimal if and only if it is not preserved under translation by a subtorus of $T$. This last property holds if and only if $\Trop(X) \subset N_{\R}$ is not preserved under translation by a rational subspace of $N_{\R}$ \cite[Lemma~5.2]{KatzPayne2011}.

Given a cone $\pC$ of $\pS_{\trop}(3,8)/L_{\R}$, denote by $X_{\pC}$ the intersection of $\overline{X}(3,8)$ with the torus orbit of $X(\pS_{\trop}(3,8)/L_{\R})$ corresponding to $\pC$. By \cite[Lemma~3.6]{HelmKatz}, there is an isomorphism
\begin{equation}
\label{eq:initStratum}
    \init_{\sw} X(3,8) \cong X_{\pC} \times (\GG_m)^{\dim \pC}
\end{equation}
where $\sw$ is any point in the relative interior of $\pC$. The  fan  $\pS_{\trop}(3,8)/L_{\R}$ is convexly disjoint by \cite[Proposition~7.9]{Schock}. This proves that the strata $X_{\pC}$ of $\overline{X}(3,8)$ which are irreducible, i.e., for cones  $\pC$ not in the $\Sn{8}$--orbit of $\spCone$ by Theorem \ref{thm:connectedInitDegIntro}, are log minimal. But the strata $X_{\pC}$ for $\pC$ in the $\Sn{8}$--orbit of $\spCone$ each consists of two disjoint points, which are individually log minimal, as required.  
\end{proof}

\bibliographystyle{abbrv}
\bibliography{bibliographie}
\label{sec:biblio}

\end{document}